\documentclass[11pt]{amsart}

\usepackage{a4wide}
\usepackage[title,titletoc,toc]{appendix}
\usepackage[T1]{fontenc}
\usepackage[applemac]{inputenc}
\usepackage{amsmath}
\usepackage{amssymb}
\usepackage{amsfonts}
\usepackage[english]{babel}
\usepackage{amsthm}
\usepackage{enumerate}
\usepackage{bbm}
\usepackage{pdfsync}
\usepackage{graphicx,color}
\usepackage[all]{xy}

\usepackage{caption}
\usepackage{epsfig}
\graphicspath{{./}}
\usepackage{epstopdf}
\usepackage{algorithmic}
\ifpdf
  \DeclareGraphicsExtensions{.eps,.pdf,.png,.jpg}
\else
  \DeclareGraphicsExtensions{.eps}
\fi

\usepackage[top=3cm, bottom=2cm, left=3cm, right=3cm]{geometry}
\usepackage[foot]{amsaddr}

\newtheorem{theorem}{Theorem}
\newtheorem{definition}[theorem]{Definition}
\newtheorem{proposition}[theorem]{Proposition}
\newtheorem{lemma}[theorem]{Lemma}
\newtheorem{corollary}[theorem]{Corollary}
\newtheorem{remark}[theorem]{Remark}

\numberwithin{equation}{section}

\newcommand{\RR}{\mathbb{R}}
\newcommand{\N}{\mathbb{N}}

\let\SS\S

\newcommand{\AC}{\mathrm{ac}}
\newcommand{\eps}{\varepsilon}

\newcommand{\supp}{{\rm supp\ }}

\newcommand{\mW}{{\mathcal W}}

\newcommand{\mF}{{\mathcal F}}

\newcommand{\mH}{{\mathcal H}}
\newcommand{\mY}{{\mathcal Y}}

\newcommand{\mL}{{\mathcal L}}

\newcommand{\bC}{{\mathbbm{1}}}

\newcommand{\bla}{\big\langle}
\newcommand{\bra}{\big\rangle}

\title{Ground States in the Diffusion-Dominated Regime}

\author{Jos\'e A. Carrillo$^1$}
\address{$^1$Department of Mathematics, Imperial College London, South Kensington Campus, London SW7 2AZ, UK. Email: \texttt{carrillo@imperial.ac.uk}.}

\author{Franca Hoffmann$^{2,1}$}
\address{$^2$DPMMS, Centre for Mathematical Sciences, University of Cambridge, Wilberforce Road, Cambridge CB3 0WA, UK. Email: \texttt{fkoh2@cam.ac.uk}.}

\author{Edoardo Mainini$^3$}
\address{$^3$ Dipartimento di Ingegneria Meccanica, Universit\`a degli Studi di Genova, Piazzale Kennedy, Pad. D, 16129, Genova,
          Italia. Email: \texttt{edoardo.mainini@unipv.it}.}

\author{Bruno Volzone$^4$}
\address{$^4$ Dipartimento di Ingegneria Universit\`a degli Studi di Napoli ``Parthenope'', Napoli, 80143,
          Italia. Email: \texttt{bruno.volzone@uniparthenope.it}.}

\begin{document}

\maketitle

\begin{abstract}
We consider macroscopic descriptions of particles where repulsion is modelled by non-linear power-law diffusion and attraction by a homogeneous singular
kernel leading to variants of the Keller-Segel model of chemotaxis.
We analyse the regime in which diffusive forces are stronger than attraction between particles, known as the diffusion-dominated regime, and show that all stationary states of the system are radially symmetric decreasing and compactly supported. The model can be
formulated as a gradient flow of a free energy functional for which the overall convexity properties are not known. We show that global minimisers of the
free energy always exist. Further, they are radially symmetric, compactly supported, uniformly bounded and $C^\infty$ inside their support. Global minimisers enjoy certain regularity
properties if the diffusion is not too slow, and in this case, provide stationary states of the system.  In one dimension, stationary states are
characterised as optimisers of a functional inequality which establishes equivalence between global minimisers and stationary states, and allows to deduce uniqueness.
\end{abstract}


 \section{Introduction}\label{sec:introdiffdom}

We are interested in the diffusion-aggregation equation
\begin{equation}\label{eq:KS}
 \partial_t \rho = \Delta \rho^m +\chi\nabla \cdot \left( \rho \, \nabla S_k\right)
\end{equation}
for a density $\rho(t,x)$ of unit mass defined on $\RR_+ \times \RR^N$, and where we define the mean-field potential $S_k(x) := W_k(x)\ast\rho(x)$ for some
interaction kernel $W_k$. The parameter $\chi>0$ denotes the interaction strength. Since \eqref{eq:KS} conserves mass, is positivity preserving and invariant by translations, we work with solutions
$\rho$ in the set
\begin{equation*}
\mY:=\left\{ \rho \in L_+^1(\RR^N) \cap L^m(\RR^N)\,,\,||\rho||_1=1\, ,\, \int_{\RR^N} x\rho(x)\,dx=0\right\}\, .
\end{equation*}
The interaction $W_k$ is given by the Riesz kernel
\begin{equation*}
 W_k(x)=\frac{|x|^{k}}{k}, \quad k \in (-N,0). 
\end{equation*}
Let us write $k=2s-N$ with $s\in \left(0,\frac{N}{2}\right)$. Then the convolution term $S_k$ is governed by a fractional diffusion process,
$$
c_{N,s}(-\Delta)^s S_k =  \rho\, , \qquad c_{N,s}=(2s-N) \frac{\Gamma\left(\frac{N}{2}-s\right)}{\pi^{N/2}4^s\Gamma(s)}
=\frac{k\Gamma\left(-k/2\right)}{\pi^{N/2}2^{k+N}\Gamma\left(\frac{k+N}{2}\right)}\, .
$$
For $k>1-N$ the gradient $\nabla S_k:= \nabla \left(W_k \ast \rho\right)$ is well defined locally.
For $k \in \left(-N,1-N\right]$ however, it becomes a singular integral, and we thus define it via a Cauchy principal value,
\begin{equation}\label{gradS}
 \nabla S_k(x) :=
 \begin{cases}
  \nabla \left(W_k \ast \rho\right)(x)\, ,
  &\text{if} \, \, 1-N<k<0\, , \\[2mm]
  \displaystyle\int_{\RR^N} \nabla W_k(x-y)\left(\rho(y)-\rho(x)\right)\, dy\, ,
  &\text{if} \, \, -N<k\leq 1-N\, .
 \end{cases}
\end{equation}
Here, we are interested in the porous medium case $m>1$ with $N \geq 3$. The corresponding energy functional writes
\begin{equation}\label{eq:functional}
 \mF[\rho]= \mH_m[\rho]+\chi\mW_k[\rho]
\end{equation}
with
\begin{equation*}
 \mH_m[\rho]=\frac{1}{m-1} \int_{\RR^N} \rho^m(x)\, dx\, , \qquad
 \mW_k[\rho]=\frac{1}{2}\iint_{\RR^N \times \RR^N} \frac{|x-y|^{k}}{k}\rho(x)\rho(y)\, dxdy\, .
\end{equation*}
Given $\rho \in \mY$, we see that $\mH_m$ and $\mW_k$ are homogeneous by taking dilations $\rho^\lambda(x):=\lambda^N \rho(\lambda x)$. More precisely, we
obtain
\begin{equation*}
 \mF[\rho^\lambda]=\lambda^{N(m-1)}\mH_m[\rho]+\lambda^{-k}\chi\mW_k[\rho]\,.
\end{equation*}
In other words, the diffusion and aggregation forces are in balance if $N(m-1)=-k$. This is the case for choosing the critical diffusion exponent $m_c:=1-k/N$ called the \emph{fair-competition regime}. In the \emph{diffusion-dominated regime} we choose $m>m_c$, which means that the diffusion part of the functional \eqref{eq:functional} dominates as $\lambda \to \infty$. In other words, concentrations are not energetically favourable for any value of $\chi>0$ and $m>m_c$. The range $0<m<m_c$ is referred to as the \emph{attraction-dominated regime}. In this work, we focus on the diffusion-dominated regime $m>m_c$.

Further, we define below the diffusion exponent $m^*$ that will play an important role for the regularity properties of global minimisers of $\mF$:
 \begin{equation}\label{mstar}
  m^*:=
  \begin{cases}
  \frac{2-k-N}{1-k-N}\, , \qquad &\text{if} \quad N\geq 1 \quad \text{and} \quad -N<k<1-N\, , \\
  + \, \infty &\text{if} \quad N\geq 2 \quad \text{and} \quad 1-N\leq k <0\, .
  \end{cases}
 \end{equation}

The main results in this work are summarised in the following two theorems:
\begin{theorem}\label{thm:main1}
Let $N\geq 1$, $\chi>0$ and $k \in (-N,0)$. All stationary states of equation \eqref{eq:KS} are radially symmetric decreasing. If $m>m_c$, then there exists a global minimiser $\rho$ of $\mF$ on $\mathcal{Y}$. Further, all global minimisers $\rho \in \mY$ are radially symmetric non-increasing, compactly supported, uniformly bounded and $C^{\infty}$ inside their support. Moreover, all global minimisers of $\mF$ are stationary states of \eqref{eq:KS},  according to Definition \ref{def:sstates},  whenever $m_c<m < m^*$.
Finally, if   $m_c<m\le2$,   we have $\rho \in
\mW^{1,\infty}\left(\RR^N\right)$.
\end{theorem}

\begin{theorem}\label{thm:main2}
 Let $N=1$, $\chi>0$, $k \in (-1,0)$ and $m>m_c$.  All stationary states of \eqref{eq:KS} are global minimisers of the energy functional $\mF$ on $\mY$. Further, stationary states of \eqref{eq:KS} in $\mY$ are unique.
\end{theorem}

\begin{center}
 \vspace{0.1cm}
\def\svgwidth{400pt}
\begingroup%
  \makeatletter%
  \providecommand\color[2][]{%
    \errmessage{(Inkscape) Color is used for the text in Inkscape, but the package 'color.sty' is not loaded}%
    \renewcommand\color[2][]{}%
  }%
  \providecommand\transparent[1]{%
    \errmessage{(Inkscape) Transparency is used (non-zero) for the text in Inkscape, but the package 'transparent.sty' is not loaded}%
    \renewcommand\transparent[1]{}%
  }%
  \providecommand\rotatebox[2]{#2}%
  \ifx\svgwidth\undefined%
    \setlength{\unitlength}{451.474408bp}%
    \ifx\svgscale\undefined%
      \relax%
    \else%
      \setlength{\unitlength}{\unitlength * \real{\svgscale}}%
    \fi%
  \else%
    \setlength{\unitlength}{\svgwidth}%
  \fi%
  \global\let\svgwidth\undefined%
  \global\let\svgscale\undefined%
  \makeatother%
  \begin{picture}(1,0.7552624)%
    \put(0,0){\includegraphics[width=\unitlength]{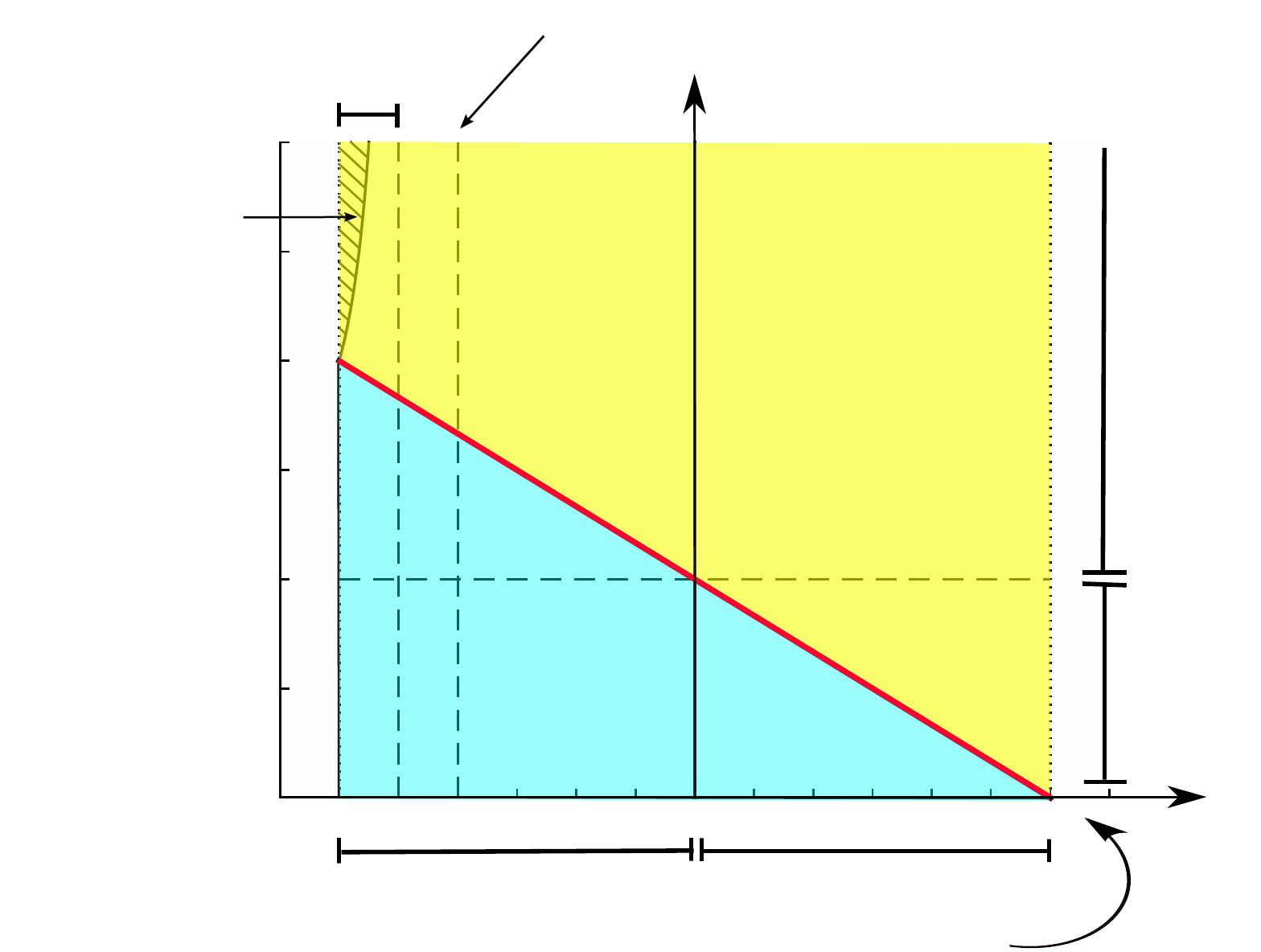}}%
    \put(0.56651897,0.36603906){\color[rgb]{0,0,0}\rotatebox{-33.88550473}{\makebox(0,0)[b]{\smash{diffusion-dominated regime}}}}%
    \put(0.76324029,0.54745823){\color[rgb]{0,0,0}\makebox(0,0)[lt]{\begin{minipage}{0.35745097\unitlength}\centering porous\\  medium\\  regime\end{minipage}}}%
    \put(0.74451336,0.27061936){\color[rgb]{0,0,0}\makebox(0,0)[lt]{\begin{minipage}{0.40113897\unitlength}\centering fast\\  diffusion\\  regime\end{minipage}}}%
    \put(0.287187,0.06848074){\color[rgb]{0,0,0}\makebox(0,0)[lt]{\begin{minipage}{0.22073554\unitlength}\centering $W_k$ singular\end{minipage}}}%
    \put(0.55924824,0.06833508){\color[rgb]{0,0,0}\makebox(0,0)[lt]{\begin{minipage}{0.2649428\unitlength}\centering $W_k$ non-singular\end{minipage}}}%
    \put(0.26460525,0.09672648){\color[rgb]{0,0,0}\makebox(0,0)[b]{\smash{-N}}}%
    \put(0.27074102,0.11241406){\color[rgb]{0,0,0}\makebox(0,0)[lt]{\begin{minipage}{0.08485928\unitlength}\centering 1-N\end{minipage}}}%
    \put(0.32975351,0.11260777){\color[rgb]{0,0,0}\makebox(0,0)[lt]{\begin{minipage}{0.07783039\unitlength}\centering 2-N\end{minipage}}}%
    \put(0.52476866,0.11319501){\color[rgb]{0,0,0}\makebox(0,0)[lt]{\begin{minipage}{0.05509688\unitlength}\centering 0\end{minipage}}}%
    \put(0.79511893,0.11506601){\color[rgb]{0,0,0}\makebox(0,0)[lt]{\begin{minipage}{0.06886259\unitlength}\centering N\end{minipage}}}%
    \put(0.01438447,0.71194821){\color[rgb]{0,0,0}\makebox(0,0)[lt]{\begin{minipage}{0.53497704\unitlength}\centering $\nabla W_k \notin L^1_{loc }(\RR^N)$\end{minipage}}}%
    \put(0.931935,0.11704747){\color[rgb]{0,0,0}\makebox(0,0)[lt]{\begin{minipage}{0.08665337\unitlength}\centering $k$\end{minipage}}}%
    \put(0.51868103,0.71415717){\color[rgb]{0,0,0}\makebox(0,0)[lt]{\begin{minipage}{0.10965796\unitlength}\centering $m$\end{minipage}}}%
    \put(0.16667725,0.76115049){\color[rgb]{0,0,0}\makebox(0,0)[lt]{\begin{minipage}{0.53653914\unitlength}\centering $W_{2-N} =$ Newtonian potential\end{minipage}}}%
    \put(0.3005613,0.03056526){\color[rgb]{0.0745098,0.05882353,0.05882353}\makebox(0,0)[lt]{\begin{minipage}{0.50608062\unitlength}\centering fair-competition regime $m_c=1-k/N$\end{minipage}}}%
    \put(0.17136318,0.30349125){\color[rgb]{0,0,0}\makebox(0,0)[lt]{\begin{minipage}{0.0413924\unitlength}\centering 1\end{minipage}}}%
    \put(0.17370616,0.13714071){\color[rgb]{0,0,0}\makebox(0,0)[lt]{\begin{minipage}{0.03592548\unitlength}\centering 0\end{minipage}}}%
    \put(0.17292516,0.47921363){\color[rgb]{0,0,0}\makebox(0,0)[lt]{\begin{minipage}{0.03436348\unitlength}\centering 2\end{minipage}}}%
    \put(-0.01607404,0.5994859){\color[rgb]{0,0,0}\makebox(0,0)[lt]{\begin{minipage}{0.24210643\unitlength}\centering $m^*=\frac{2-k-N}{1-k-N}$\end{minipage}}}%
    \put(0.47248373,0.2465821){\color[rgb]{0,0,0}\rotatebox{-32.02922618}{\makebox(0,0)[b]{\smash{attraction-dominated regime}}}}%
  \end{picture}%
\endgroup%

 \vspace{-0.2cm}
 
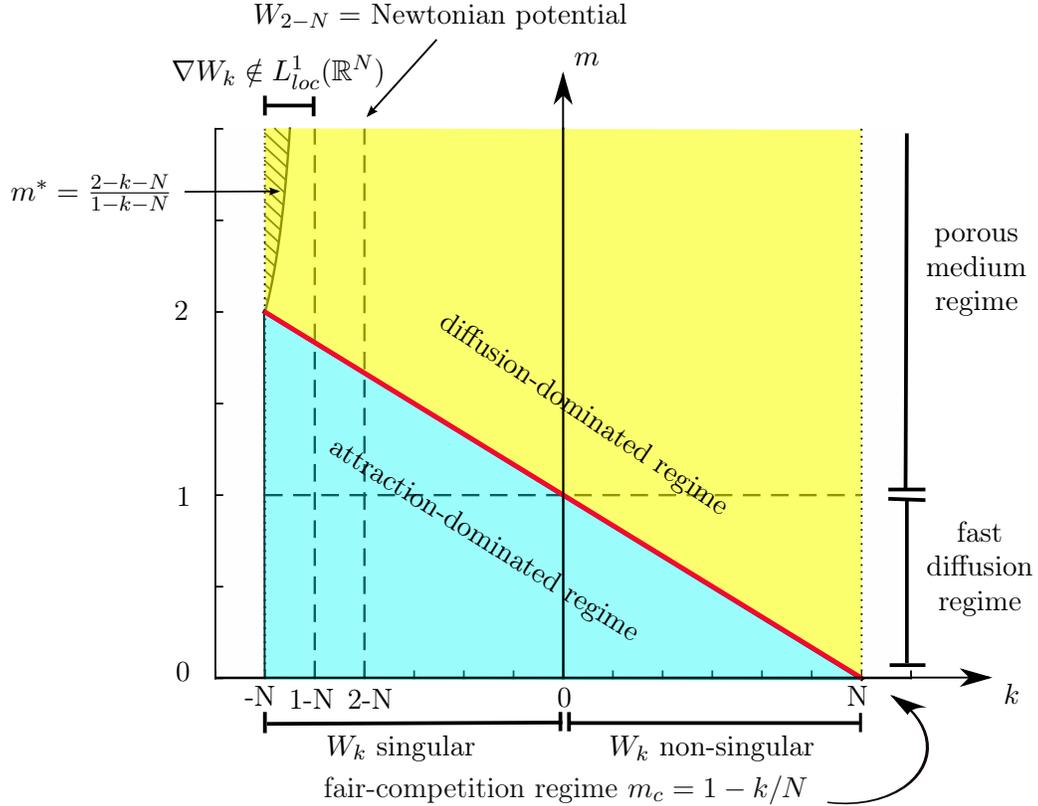
\captionof{figure}{Overview of the parameter space $(k,m)$ for $N\geq 3$: fair-competition regime ($m=m_c$, red line), diffusion-dominated regime ($m>m_c$, yellow region) and attraction-dominated regime ($m<m_c$, blue region).
 For $m=m_c$, attractive and repulsive forces are in balance (i.e. in \emph{fair competition}).
 For $m_c<m<m^*$ in the diffusion-dominated regime, global minimisers of $\mF$ are stationary states of \eqref{eq:KS}, see Theorem \ref{thm:main1}, a result which we are not able to show for $m\ge m^*$ (striped region).
 \vspace{-0.1cm}}
\label{fig:overviewregimes}
\end{center}

Aggregation-diffusion equations of the form \eqref{eq:KS} are ubiquitous as macroscopic models of cell motility due to cell adhesion and/or chemotaxis phenomena while taking into account volume filling constraints \cite{HiPai01,PaiHi02,CaCa06}. The non-linear diffusion models the very strong localised repulsion between cells while the attractive non-local term models either cell movement toward chemosubstance sources or attractive interaction between cells due to cell adhension by long filipodia. They encounter applications in cancer invasion models, organogenesis and pattern formation \cite{GC,DTGC,PBSG,MT15,CHS17}.

The archetypical example of the Keller-Segel model in two dimensions corresponding to the logarithmic case $(m = 1,k=0)$ has been deeply studied by many authors \cite{KeSe70,KeSe71a,Nanjundiah73,JaLu92,Nagai95,DoPe04,BlaDoPe06,Perthame06,BlaCaMa07,BCC,BCC12,CF,CLM14}, although there are still plenty of open problems. In this case, there is an interesting dichotomy based on a critical parameter $\chi_c>0$: the density exists globally in time if $0<\chi<\chi_c$ (diffusion overcomes self-attraction) and expands self-similarly \cite{CaDo14,EM16}, whereas blow-up occurs in finite time when $\chi>\chi_c$ (self-attraction overwhelms diffusion), while for $\chi=\chi_c$ infinitely many stationary solutions exist with intricated basins of attraction \cite{BCC12}. The three-dimensional configuration with Newtonian interaction $(m = 1,k = 2-N)$ appears in gravitational physics \cite{CLM,CM}, although it does not have this dichotomy, belonging to the attraction-dominated regime. However, the dichotomy does happen for the particular exponent $m=4/3$ of the non-linear diffusion for the 3D Newtonian potential as discovered in \cite{BCL}. This was subsequently generalised for the fair-competition regime where $m=m_c$ for a given $k\in(-N,0)$ in \cite{CCH1,CCHCetraro}.

In fact, as mentioned before two other different regimes appear: the diffusion-dominated case when $m>m_c$ and the attraction-dominated case when $m<m_c$. In Figure \ref{fig:overviewregimes}, we make a sketch of the different regimes including cases related to non-singular kernels for the sake of completeness. Note that non-singular kernels $k>0$ allow for values of $m<1$ corresponding to fast-diffusion behaviour in the diffusion-dominated regime $m>m_c$. We refer to \cite{CCH1,CCHCetraro} and the references therein for a full discussion of the state of the art in these regimes.

In the diffusion-dominated case, it was already proven in \cite{CCV} that global minimisers exist in the particular case of $m>1=m_c$ for the logarithmic interaction kernel $k=0$. Their uniqueness up to translation and mass normalisation is a consequence of the important symmetrisation result in \cite{CHVY} asserting that all stationary states to \eqref{eq:KS} for $2-N\leq k<0$ are radially symmetric. We will generalise this result to our present framework for the range $-N<k<2-N$ not included in \cite{CHVY} due to the special treatment needed for the arising singular integral terms. This is the main goal of Section \ref{sec:sstatesdiffdom} where we remind the reader the precise definition and basic properties of stationary states for \eqref{eq:KS}. In short, we show that stationary solutions are continuous compactly supported radially non-increasing functions with respect to their centre of mass. Some of these results are in fact generalisations of previous results in \cite{CCH1,CHVY} and we skip some of the details.

Let us finally comment that the symmetrisation result reduces the uniqueness of stationary states to uniqueness of radial stationary states that eventually leads to a full equivalence between stationary states and global minimisers of the free energy \eqref{eq:functional}. This was used in \cite{CHVY} to solve completely the 2D case with $m>1=m_c$ for the logarithmic interaction kernel $k=0$, and it was the new ingredient to fully characterise the long-time asymptotics of \eqref{eq:KS} in that particular case.

In view of the main results already announced above, we show in Section \ref{sec:minsdiffdom} the existence of global minimisers for the full range $m>m_c$ and $k\in (-N,0)$ which are steady states of the equation \eqref{eq:KS} as soon as $m<m^*$. This additional constraint on the range of non-linearities appears only in the most singular range $-N<k<1-N$ and allows us to get the right H\"older regularity on the minimisers in order to make sense of the singular integral in the gradient of the attractive non-local potential force \eqref{gradS}. 

Besides existence of minimisers, Section \ref{sec:minsdiffdom} contains some of the main novelties of this paper. First, in order to prove boundedness of minimisers, we develop  a fine estimate on the interaction term based on the asymptotics of the Riesz potential of radial functions, and show that this estimate is well suited exactly for the diffusion dominated regime (see Lemma \ref{Newtonlemma} and Theorem \ref{prop:Linfty}).
Moreover, thanks to the Schauder estimates for the fractional Laplacian, we improve the regularity results for minimisers in \cite{CCH1} and show that they are smooth inside their support, see Theorem \ref{cor:Cinfty}.  This result applies both to the diffusion dominated and fair competition regime.

These global minimisers are candidates to play an important role in the long-time asymptotics of \eqref{eq:KS}. We show their uniqueness in one dimension by optimal transportation techniques in Section \ref{sec:!diffdom}. The challenging open problems remaining are uniqueness of radially non-increasing stationary solutions to \eqref{eq:KS} in its full generality and the long-time asymptotics of \eqref{eq:KS} in the whole diffusion-dominated regime, even for non-singular kernels within the fast diffusion case.

Plan of the paper: In Section \ref{sec:sstatesdiffdom} we define and analyse stationary states, showing that they are radially symmetric and compactly supported.
Section \ref{sec:minsdiffdom} is devoted to global minimisers. We show that global minimisers exist, are bounded and we provide their regularity properties.
Eventually, Section \ref{sec:!diffdom} proves uniqueness of stationary states in the one-dimensional case.

\section{Stationary states}\label{sec:sstatesdiffdom}

Let us define precisely the notion of stationary states to the diffusion-aggregation equation \eqref{eq:KS}.

\begin{definition}\label{def:sstates}
 Given $\bar \rho \in L_+^1\left(\RR^N\right) \cap L^\infty\left(\RR^N\right)$ with $||\bar \rho||_1=1$   and letting $\bar S_k=W_k\ast\bar\rho$, we say that $\bar\rho$     is a \textbf{stationary state} for the
 evolution equation \eqref{eq:KS} if $\bar \rho^{m} \in \mW_{loc}^{1,2}\left(\RR^N\right)$, $\nabla \bar S_k\in L^1_{loc}\left(\RR^N\right)$, and
 it satisfies
 \begin{equation}\label{eq:sstates}
  \nabla \bar \rho^m=- \chi\, \bar \rho \nabla \bar S_k
 \end{equation}
in the sense of distributions in $\RR^N$.
If $-N<k\leq 1-N$, we further require $\bar \rho \in C^{0,\alpha}\left(\RR^N\right)$   for some   $\alpha \in (1-k-N,1)$.
\end{definition}

In fact, as shown in \cite{CCH1} via a near-far field decomposition argument of the drift term, the function $S_k$ and its gradient defined in \eqref{gradS} satisfy even more than the regularity $\nabla S_k \in
L_{loc}^1\left(\RR^N\right)$ required in Definition \ref{def:sstates}:

\begin{lemma}\label{lem:regS} Let $\rho \in L_+^1\left(\RR^N\right)\cap L^\infty\left(\RR^N\right)$ with $||\rho||_1=1$ and $k \in (-N,0)$. Then the
following regularity properties hold:
\begin{enumerate}[(i)]
\item $ S_k \in L^{\infty}\left(\RR^N\right)$.
\item $\nabla  S_k \in L^{\infty}\left(\RR^N\right)$, assuming additionally $\rho \in C^{0,\alpha}\left(\RR^N\right)$ with $\alpha \in (1-k-N,1)$
    in the range $k \in (-N,1-N]$.
\end{enumerate}
\end{lemma}

Lemma \ref{lem:regS} implies further regularity properties for stationary states of \eqref{eq:KS}. For precise proofs, see \cite{CCH1}.

\begin{proposition} \label{prop:sstatesreg}
Let $k \in (-N,0)$ and $m>m_c$. If $\bar \rho$ is a stationary state of equation \eqref{eq:KS}   and $\bar S_k=W_k\ast\bar\rho$,   then $\bar \rho$ is continuous on $\RR^N$, $\bar \rho^{m-1} \in \mW^{1,\infty}\left(\RR^N\right)$, and
  \begin{equation}\label{eq:EL}
 \bar \rho(x)^{m-1} = \frac{m-1}{m} \left( C[\bar\rho](x)- \chi\bar S_k(x)\right)_+\, , \qquad
 \forall \, x \in \RR^N\, ,
\end{equation}
 where $C[\bar\rho](x)$ is constant on each connected component of $\supp(\bar \rho)$.
\end{proposition}

It follows from Proposition \ref{prop:sstatesreg} that $\bar \rho \in \mW^{1,\infty}\left(\RR^N\right)$ in the case   $m_c<m\le2$.

\


\subsection{Radial Symmetry of Stationary States}
The aim of this section is to prove that stationary states of \eqref{eq:KS} are radially symmetric.
This is one of the main results of \cite{CHVY}, and is achieved there under the assumption that the interaction kernel is not more singular than the Newtonian potential close to the origin.
As we will briefly describe in the proof of the next result, the main arguments continue to hold even for the more singular Riesz kernels $W_{k}$.
\begin{theorem}[Radiality of stationary states]\label{thm:radiality} Let $\chi>0$ and $m>m_c$. If   $\bar\rho \in L^1_+(\RR^N) \cap L^\infty(\RR^N)$ with $\|\bar\rho\|_1=1$   is a stationary state of \eqref{eq:KS} in the sense of Definition
\ref{def:sstates}, then $\bar\rho$ is radially symmetric   non-increasing up to a translation.
\end{theorem}
\begin{proof}
The proof is based on a contradiction argument, being an adaptation of that in \cite[Theorem 2.2]{CHVY}, to which we address
the reader the more technical details. Assume that $\bar\rho$ is \emph{not} radially decreasing up to \emph{any} translation. By
Proposition \ref{prop:sstatesreg}, we have
\begin{equation}
\left|\nabla \bar \rho^{m-1}(x)\right|\leq c \label{boundgradpowerrho}
\end{equation}
for some positive constant $c$ in $\text{supp}(\bar\rho)$. Let us now introduce the \emph{continuous Steiner symmetrisation}  $S^\tau \bar\rho$ in direction $e_1 =
(1,0,\cdots,0)$ of $\bar\rho$ as follows. For any $x_1 \in \mathbb{R}, x'\in \mathbb{R}^{N-1}, h>0$, let
\begin{equation*}
S^\tau\bar\rho(x_1, x') := \int_0^\infty \bC_{M^\tau(U_{x'}^h)}(x_1) \,dh\,,
\end{equation*}
where
\[
U_{x'}^h = \{x_1 \in \mathbb{R}: \bar\rho(x_1, x')>h\}
\]
and $M^\tau(U_{x'}^h)$ is the continuous Steiner symmetrisation of the $U_{x'}^h$ (see \cite{CHVY} for the precise definitions and all the related properties). As in \cite{CHVY}, our aim is to show that there exists a continuous family of functions $\mu(\tau,x)$ such that $\mu(0,\cdot)=\bar\rho$ and some positive constants $C_{1}>0$, $c_{0}>0$ and
a small $\delta_{0}>0$ such that the following estimates hold for all $\tau\in[0,\delta_{0}]$:
\begin{equation}
\mF[\mu(\tau, \cdot)]-\mF[\bar\rho]\leq-c_{0}\tau\label{1ineq}
\end{equation}
\begin{equation}
|\mu(\tau,x)-\bar\rho(x)|\leq C_{1}\bar{\rho}(x)\tau\quad\quad\text{for all } x\in\RR^{N}\label{2ineq}
\end{equation}
\begin{equation}
\int_{\Omega_{i}}\left(\mu(\tau,x)-\bar\rho(x)\right)dx=0\quad\quad\text{for any connected component } \Omega_{i}\text { of }\text{supp}(\bar\rho).\label{3ineq}
\end{equation}
Following the arguments of the proof in \cite[Proposition 2.7]{CHVY}, if we want to construct a continuous family $\mu(\tau,\cdot)$ for \eqref{2ineq} to hold, it is convenient to modify suitably the continuous Steiner symmetrisation $S^\tau \bar\rho$ in order to have a better control of the speed in which the level sets $U_{x'}^h$  are moving. More precisely, we define $\mu(\tau,\cdot) = \tilde S^\tau \bar\rho$ as
\[
\tilde S^\tau\bar\rho_0(x_1, x') := \int_0^\infty \bC_{M^{v(h)\tau}(U_{x'}^h)}(x_1)\, dh
\]
with $v(h)$ defined as
\begin{equation*}
\label{def:v}
v(h) := \begin{cases}
1 & h> h_0\,,\\
0 & 0<h\leq h_0\,,
\end{cases}
\end{equation*}
for some sufficiently small constant $h_0>0$ to be determined. Note that this choice of the velocity is different to the one in \cite[Proposition 2.7]{CHVY} since we are actually keeping the level sets of $\tilde S^\tau \bar\rho(\cdot,x^{\prime})$ frozen below the layer at height $h_{0}$.  Next, we note that inequality \eqref{boundgradpowerrho} and the Lipschitz regularity of $\bar{S}_{k}$ (Lemma \ref{lem:regS}) are the only basic ingredients used in the proof of \cite[Proposition 2.7]{CHVY} to show that the family $\mu(\tau,\cdot)$ satisfies  \eqref{2ineq} and \eqref{3ineq}. Therefore, it remains to prove \eqref{1ineq}. Since different level sets of $\tilde S^\tau \bar\rho(\cdot,x^{\prime})$  are moving at different speeds $v(h)$, we do not have
$M^{v(h_1)\tau}(U_{x'}^{h_1}) \subset M^{v(h_2)\tau}(U_{x'}^{h_2}) $ for all $h_1>h_2$, but it is still possible to prove that (see \cite[Proposition 2.7]{CHVY})
\begin{equation*}
\label{eq:s_dec}
\mH_m[\tilde S^\tau\bar\rho] \leq \mH_m[\bar\rho] \text{ for all }\tau\geq 0.
\end{equation*}
Then, in order to establish \eqref{1ineq}, it is enough to
show
\begin{equation}\label{eq:goal0}
\mW_k[\tilde S^\tau\bar\rho] \leq \mW_k[\bar\rho] -\chi c_0 \tau \quad\text{ for all $\tau\in[0,\delta_0]$}, \text{ for some $c_0>0$ and $\delta_0>0$.}
\end{equation}
As in the proof of \cite[Proposition 2.7]{CHVY}, proving \eqref{eq:goal0} reduces to show that for sufficiently small $h_0>0$ one has
\begin{equation}
\label{eq:goal1}
\left|\mW_k[\tilde S^\tau\bar\rho] - \mW_k[ S^\tau\bar\rho]\right| \leq \frac{c\chi\tau}{2}\quad \text{ for all $\tau$}.
\end{equation}
To this aim, we write
\begin{equation*}
\label{eq:split1}
S^\tau\bar\rho(x_1, x') = \int_{h_0}^\infty \bC_{M^\tau(U_{x'}^h)}(x_1)dh + \int_{0}^{h_0} \bC_{M^\tau(U_{x'}^h)}(x_1)dh =: f_1(\tau,x)+f_2(\tau,x)
\end{equation*}
and we split $\tilde S^\tau \bar\rho$ similarly, taking into account that $v(h)=1$ for all $h>h_0$:
\begin{equation*}
\label{eq:split2}
\tilde S^\tau \bar\rho(x_1, x') = f_1(\tau,x) +  \int_{0}^{h_0} \bC_{M^{v(h)\tau}(U_{x'}^h)}(x_1)dh =: f_1(\tau,x) + \tilde f_2(\tau,x).
\end{equation*}
Note that
\[
f_{2}=S^{\tau}(\mathcal{T}^{h_{0}}\bar\rho)\, ,
\]
where $\mathcal{T}^{h_{0}}\bar\rho$ is the truncation at height $h_{0}$ of $\bar\rho$. Since $v(h)=0$ for $h\leq h_{0}$, we have
\[
\tilde f_2=\mathcal{T}^{h_{0}}\bar\rho.
\]
If we are in the singular range $k\in (-N,1-N]$, we have $\bar\rho \in C^{0,\alpha}\left(\RR^N\right)$ for some $\alpha \in (1-k-N,1)$. Since the continuous Steiner symmetrisation decreases the modulus of continuity (see \cite[Theorem 3.3]{Brockpolar} and \cite[Corollary 3.1]{Brockpolar}), we also have $S^\tau \bar\rho,\, f_{2},\,\tilde f_2\in C^{0,\alpha}\left(\RR^N\right)$. Further, Lemma \ref{lem:regS} and the arguments of \cite[Proposition 2.7]{CHVY} guarantee that the expressions
$$
 A_1(\tau):=\left|\int f_2 (W_k \ast f_1)-\tilde f_2 (W_k \ast f_1)dx\right|
 \quad\text{and}\quad
  A_2(\tau):=\left|\int f_2 (W_k \ast f_2)-\tilde f_2 (W_k \ast \tilde f_2)dx\right|
$$
can be controlled by $||\bar \rho||_\infty$ and the $\alpha$-H\"{o}lder seminorm of $\bar \rho$.
Hence, we can apply the argument in \cite[Proposition 2.7]{CHVY} to conclude for the estimate \eqref{eq:goal1}. Now it is possible to proceed exactly as in the proof of  \cite[Theorem 2.2]{CHVY} to show that for some positive constant $C_{2}$, we have the quadratic estimate
\[
\left|\mF[\mu(\tau,\cdot)]-\mF[\bar\rho]\right|\leq C_{2}\tau^{2}\, ,
\]
which is a contradiction with \eqref{1ineq} for small $\tau$.
\end{proof}

\subsection{Stationary States are Compactly Supported}
In this section, we will prove that all stationary states of equation \eqref{eq:KS} have compact support, which agrees with the properties shown in \cite{KY,CCV,CHVY}. We begin by stating a useful asymptotic estimate on the Riesz potential inspired by \cite[\SS 4]{TS}. For the proof of Proposition \ref{prop:rieszestimate}, see Appendix \ref{sec:Rieszestimate}.
\begin{proposition}[Riesz potential estimates]\label{prop:rieszestimate}
Let $k \in (-N,0)$ and let $\rho\in \mY$ be radially symmetric.
\begin{enumerate}[(i)]
 \item If $1-N<k<0$, then $|x|^k\ast \rho(x)\le C_1 |x|^{k}$ on $\RR^N$.
 \item If $-N<k\leq 1-N$ and if $\rho$ is supported on a ball $B_R$ for some $R<\infty$, then
 \[
|x|^k\ast \rho(x)\le C_2 T_k(|x|,R)\, |x|^{k}\, , \qquad \forall\, \,  |x|>R\, ,
\]
where
\begin{equation}\label{Pik}
T_k(|x|,R):=\left\{\begin{array}{ll}
 \left(\frac{|x|+R}{|x|-R}\right)^{1-k-N}\quad&\mbox{if $k\in(-N,1-N)$},\\\\
 \left(1+\log\left(\frac{|x|+R}{|x|-R}\right)\right)\quad&\mbox{if $k=1-N$}
\end{array}\right.
\end{equation}
\end{enumerate}
Here, $C_1>0$ and $C_2>0$ are explicit   constants   depending only on $k$ and $N$.
\end{proposition}
From the above estimate, we can derive the expected asymptotic behaviour at infinity.
\begin{corollary}\label{cor:simpledecay}
Let $\rho\in \mY$ be radially   non-increasing.  Then $W_k\ast\rho$ vanishes at infinity, with decay  not
faster than that of $|x|^{k}$.
\end{corollary}
\begin{proof}
Notice that Proposition \ref{prop:rieszestimate}(i) entails the decay of the Riesz potential at infinity for $1-N<k<0$. Instead, let $-N<k\leq 1-N$. Let $r\in
(1-k-N,1)$ and notice that
$|y|^{k}\le |y|^{k+r}$ if $|y|\ge 1$, so that if $B_1$ is the unit ball centered at the origin we have
\[\begin{aligned}
|x|^k\ast\rho(x)&\le \int_{B_{1}}\rho(x-y)|y|^{k}\,dy+\int_{B_1^C}\rho(x-y)|y|^{k+r}\,dy\\
&\le \left( \sup_{y\in B_1}\rho(x-y)\right)\int_{B_1}|y|^{k}\,dy+(W_{k+r}\ast\rho)(x).
\end{aligned}\]
The first term in the right hand side vanishes as $|x|\to\infty$, since $y\mapsto |y|^{k}$ is integrable at the origin, and since $\rho$ is radially
  non-increasing   and vanishing at infinity as well. The second term goes to zero at infinity thanks to Proposition \ref{prop:rieszestimate}(i), since the choice of $r$ yields $k+r>1-N$.

  On the other hand, the decay at infinity of the Riesz potential can not be faster than that of $|x|^{k}$. To see this, notice that
 there holds
\[
|x|^k\ast\rho(x)\ge \int_{B_1}\rho(y)|x-y|^{k}\,dy \ge (|x|+1)^{k} \int_{B_1}\rho(y)\,dy
\]
 with $\int_{B_1}\rho>0$ since $\rho\in\mY$ is radially non-increasing.
\end{proof}
As a rather simple consequence of Corollary \ref{cor:simpledecay}, we obtain:
\begin{corollary}\label{cor:compactsupp}
 Let $\bar\rho$ be a stationary state of \eqref{eq:KS}. Then $\bar\rho$ is compactly supported.
\end{corollary}
\begin{proof}
By Theorem \ref{thm:radiality} we have that $\bar\rho$ is radially   non-increasing up to a translation. Since the translation of a stationary state is itself a stationary state, we may assume that $\bar\rho$ is radially symmetric with respect to the origin.    Suppose by contradiction that $\bar\rho$ is supported on the whole of $\mathbb{R}^N$, so that equation \eqref{eq:EL} holds on the whole
$\mathbb{R}^N$, with $C_{k}[\bar\rho](x)$ replaced by a unique constant $C$. Then we necessarily have $C=0$. Indeed, $\bar\rho^{m-1}$ vanishes at infinity since it is radially decreasing and
integrable, and  by Corollary \ref{cor:simpledecay} we have that   $\bar S_{k}=W_k\ast\bar\rho$   vanishes at infinity as well. Therefore $$\bar\rho=\left(\frac{\chi(m-1)}{m}\bar S_{k}\right)^{1/(m-1)}.$$
But Corollary \ref{cor:simpledecay} shows that $W_k\ast\rho$ decays at infinity not faster than $|x|^{k}$ and this would entail, since $m>m_c$, a decay at
infinity of $\rho$ not faster than   that of $|x|^{-N}$,   contradicting the integrability of $\rho$.
\end{proof}

\section{Global Minimisers}\label{sec:minsdiffdom}

We start this section 
by recalling a key ingredient for the analysis of the regularity of the drift term in \eqref{eq:KS}, i.e.  certain functional inequalities which are variants of the Hardy-Littlewood-Sobolev (HLS) inequality, also known as the weak Young's inequality \cite[Theorem 4.3]{LiebLoss}: for all $f \in L^p(\RR^N)$, $g \in L^q(\RR^N)$ there exists an optimal constant $C_{HLS}=C_{HLS}(p,q,k)>0$ such that
\begin{align}
&\left|\iint_{\RR^N\times \RR^N} f(x) {|x-y|^{k}} g(y)\, dxdy\right| \leq C_{HLS} \|f\|_{p} \|g\|_{q}\, ,  \label{eq:HLS}\\
&\text{if}\qquad \dfrac1p + \dfrac1q  = 2 + \dfrac k N \,  , \quad p,q>1\, ,\quad   k\in(-N,0) \, . \nonumber
\end{align}
The optimal constant $C_{HLS}$ is found in \cite{Lieb83}.
In the sequel, we will make use of the following variations of above HLS inequality:
\begin{theorem}\label{thm:HLSm}
Let $k \in (-N,0)$, and $m>m_c$. For $f \in L^1(\RR^N) \cap L^m(\RR^N)$, we have
\begin{equation}\label{eq:HLSm}
\left|\iint_{\RR^N\times \RR^N} {|x-y|^{k}}f(x)f(y) dxdy\right| \leq C_* ||f||^{(k+N)/N}_1 ||f||^{m_c}_{m_c},
\end{equation}
where $C_*=C_*(k,m,N)$ is the best constant.
\end{theorem}
\begin{proof} The inequality is a direct consequence of the standard sharp HLS inequality  and of
H\"{o}lder's inequality. It follows that $C_*$ is finite and bounded from above by the optimal constant in the HLS inequality.
\end{proof}

\subsection{Existence of Global Minimisers}
\begin{theorem}[Existence of Global Minimisers]\label{thm:Emins}
 For all $\chi>0$ and $k \in (-N,0)$, there exists a global minimiser $\rho$ of $\mF$ in $\mY$. Moreover, all global minimisers of $\mF$ in $\mY$ are radially non-increasing.
\end{theorem}

We follow the concentration compactness argument as applied in Appendix A.1 of \cite{KY}. Our proof is based on \cite[Theorem II.1, Corollary II.1]{L}. Let us denote by $\mathcal{M}^p(\RR^N)$ the Marcinkiewicz space or weak $L^p$ space.
\begin{theorem}{(see \cite[Theorem II.1]{L})}\label{Lionsthm}
 Suppose $W\in \mathcal{M}^p(\RR^N)$, $1<p<\infty$, and consider the problem
 \begin{equation*}
  I_M = \inf_{\rho \in \mY_{q,M}} \left\{\frac{1}{m-1} \int_{\RR^N} \rho^m dx + \frac{\chi}{2}\iint_{\RR^N\times \RR^N} W(x-y)\rho(x)\rho(y)\,dxdy\right\}\, .
 \end{equation*}
 where
 \begin{equation*}
  \mY_{q,M}=\left\{\rho \in L^q(\RR^N)\cap L^1(\RR^N)\, , \,  \rho \geq 0\, \,  a.e., \int_{\RR^N} \rho(x) \, dx=M\right\}\, ,
  \qquad q=\frac{p+1}{p}<m\, .
 \end{equation*}
 Then there exists a minimiser of problem $(I_M)$ if the following holds:
\begin{equation}\label{cc1}
 I_{M_0} < I_{M} + I_{M_0-M}
 \qquad \text{for all} \, \, M \in (0,M_0)\, .
\end{equation}
\end{theorem}
\begin{proposition}{(see \cite[Corollary II.1]{L})}\label{Lionsprop}
 Suppose there exists some $\lambda \in (0,N)$ such that
 $$
 W(tx) \geq t^{-\lambda}W(x)
 $$
 for all $t\geq 1$. Then \eqref{cc1} holds if and only if
 \begin{equation}\label{cc2}
 I_{M} < 0
 \qquad \text{for all} \, \,M>0\, .
\end{equation}
\end{proposition}
\begin{proof}[Proof of Theorem \ref{thm:Emins}]
First of all, notice that our choice of potential $W_k(x)=|x|^{k}/k$ is indeed in $\mathcal{M}^p(\RR^N)$ with $p=-N/k$. Further, it can easily be verified
that Proposition \ref{Lionsprop} applies with $\lambda=-k$. Hence we are left to show that there exists a choice of $\rho \in \mY_{q,M}$ such that
$\mF[\rho]<0$. Let us fix $R>0$ and define
$$
\rho_*(x):=\frac{MN}{\sigma_N R^N}\,  \bC_{B_R}(x)\, ,
$$
where $B_R$ denotes the ball centered at zero and of radius $R>0$, and where $\sigma_N=2 \pi^{(N/2)}/\Gamma(N/2)$ denotes the surface area of the $N$-dimensional unit ball.
Then
\begin{align*}
 \mH_m[\rho_*]&=\frac{1}{m-1} \int_{\RR^N} \rho_*^m dx
 =\frac{(MN)^m \sigma_N^{1-m}}{N(m-1)} \, R^{N(1-m)}\, , \\
 \mW_k[\rho_*]
 &=\frac12\iint_{\RR^N\times \RR^N} W_k(x-y)\rho_*(x)\rho_*(y)\, dx dy\\
 &=  \frac{(MN)^2}{2k\sigma_N^2 R^{2N}}\iint_{\RR^N\times \RR^N}|x-y|^{k}\bC_{B_R}(x)\bC_{B_R}(y)\, dx dy \\
 & \leq \frac{(MN)^2}{2k\sigma_N^2 R^{2N}} (2R)^{k} \frac{\sigma_N^2}{N^2}R^{2N}
 = 2^{k-1}M^2 \frac{R^{k}}{k}<0\, .
\end{align*}
We conclude that
$$
\mF[\rho_*]=\mH_m[\rho_*]+\chi\mW_k[\rho_*] \leq
\frac{M^m N^{m-1} \sigma_N^{1-m}}{(m-1)} \, R^{N(1-m)} +2^{k-1}M^2\chi \frac{R^{k}}{k}\, .
$$
Since we are in the diffusion-dominated regime $N(1-m)< k<0$, we can choose $R>0$ large enough such that $\mF[\rho_*]<0$, and hence condition \eqref{cc2}
is satisfied. We conclude by Proposition \ref{Lionsprop} and Theorem \ref{Lionsthm} that there exists a minimiser $\bar \rho$ of $\mF$ in $\mY_{q,M}$ with
$q=(p+1)/p=(N-k)/N$.\\

It can easily be seen that in fact $\bar \rho \in L^m(\RR^N)$ using the HLS inequality \eqref{eq:HLS}:
 \begin{equation*}
  -\mW_k[\rho]=\frac12 \iint_{\RR^N \times \RR^N} \frac{|x-y|^{k}}{(-k)}\rho(x)\rho(y)\, dxdy
  \leq \frac{C_{HLS}}{(-2k)} ||\rho||_{r}^2\, ,
 \end{equation*}
 where $r=2N/(2N+k)=2p/(2p-1)$. Using H\"{o}lder's inequality, we find
 \begin{equation*}
 -\mW_k[\rho]\leq \frac{C_{HLS}}{(-2k)}
   ||\rho||_q^q  ||\rho||_1^{2-q}  \, .
 \end{equation*}
Hence, since $\mF[\bar \rho]<0$,
\begin{align*}
 ||\bar \rho||_m^m\leq -\chi (m-1) \mW_k[\bar \rho]
 \leq \chi (m-1)\left(\frac{M^{2-q}C_{HLS}}{(-2k)}\right) ||\bar\rho||_q^q  < \infty\, .
\end{align*}
Translating $\bar \rho$ so that its centre of mass is at zero and choosing $M=1$, we obtain a minimiser $\bar\rho$ of $\mF$ in $\mY$.
Moreover, by Riesz's rearrangement inequality \cite[Theorem 3.7]{LiebLoss}, we have
\[
\mW_k[\rho^{\#}]\leq \mW_k[\rho]\, , \qquad \forall \rho \in \mY,
\]
where $\rho^{\#}$ is the Schwarz decreasing rearrangement of $\rho$. Thus, if $\bar \rho$ is a global minimiser of $\mF$ in $\mY$, then so is $\bar\rho^{\#}$, and it follows that
\[
\mW_k[\bar \rho^{\#}]= \mW_k[\bar \rho]\, .
\]
We conclude from \cite[Theorem 3.7]{LiebLoss} that $\bar\rho=\bar\rho^{\#}$, and so all global minimisers of $\mF$ in $\mY$ are radially symmetric non-increasing.
\end{proof}

Global minimisers of $\mF$ satisfy a corresponding Euler--Lagrange
condition. The proof can be directly adapted from \cite[Theorem 3.1]{CCV} or \cite[Proposition 3.6]{CCH1}, and we omit it here.
\begin{proposition} \label{prop:characmin}
Let $k \in (-N,0)$ and $m>m_c$.
If $\rho$ is a global minimiser of the free energy functional $\mF$ in $\mY$, then $\rho$ is radially symmetric and non-increasing, satisfying
	  \begin{equation}\label{ELmin}
	   \rho^{m-1}(x) = \left(\frac{m-1}{m}\right)\,\left(D[\rho]-\chi W_k\ast \rho(x)\right)_+
	     \quad \text{a.e. in}\, \,  \RR^N.
	  \end{equation}
Here, we denote
$$
D[\rho] := 2 \mF[\rho] + \left(\frac{m-2}{m-1}\right) ||\rho||_m^m, \qquad \rho \in \mY\,.
$$
\end{proposition}

\subsection{Boundedness of Global Minimisers}
\label{sec:bddmins}

This section is devoted to showing that all global minimisers of $\mF$ in $\mY$ are uniformly bounded. In the following, for a radial function $\rho\in L^1(\mathbb{R}^N)$ we denote by $M_\rho(R):=\int_{B_R}\rho\, dx$  the corresponding mass function, where $B_R$ is a ball of radius $R$, centered at the origin.
We start with the following technical lemma:


\begin{lemma}\label{Newtonlemma}
Let $\chi>0$, $-N<k<0$, $m>1$ and $0\le q< m/N$. Assume $\rho\in \mY$ is radially decreasing.  For a fixed $H>0$,
the level set  $\{\rho\ge H\}$ is a ball centered at the origin whose radius we denote by $A_H$.
Then we have the following cross-range interaction estimate: there exists $H_0>1$, depending only on $q,N,m,\|\rho\|_m$, such that, for any $H>H_0$,
\[
\int_{B_{A_H}^C}\int_{B_{A_H}}|x-y|^{k} \rho(x)\rho(y)\,dx\,dy\le C_{k,N}\, M_{\rho}({A_H})\, \mathcal{K}_{k,q,N}(H),
\]
where
\[
\mathcal{K}_{k,q,N}(H):=\left\{\begin{array}{ll}
H^{1-q(k+N)}+ H^{-kq} \quad&\mbox{if $k\in(-N,0),\; k\neq 1-N$},\\
H^{1-q}(2+\log (1+H^q))+  H^{q(N-1)} \quad&\mbox{if $k=1-N$}
\end{array}\right.
\]
and $C_{k,N}$ is a constant depending only on $k$ and $N$.

\end{lemma}
\begin{proof}
Notice that the result is trivial if $\rho$ is bounded. The interesting case here is  $\rho$  unbounded, implying that $A_H>0$ for any $H>0$.

First of all, since $\rho \in L^m(\RR^N)$ and $\rho\ge H$ on $B_{A_H}$, the estimate
\[
\frac{\sigma_N A_H^N}{N} H^m=\int_{B_{A_H}} H^m\le \int_{B_{A_H}}\rho^m\le ||\rho||_m^m
\]
implies that $H^q A_H$ is vanishing as $H\to+\infty$ as soon as $q<m/N$, and in particular that we can find $H_0>1$, depending only on $q,m,N, ||\rho||_m$,
such that
\[
H^{-q}\ge 2 A_H\quad \mbox{for any $H>H_0$}.
\]
We fix $q\in [0,m/N)$ and $H>H_0$ as above from here on.\\

Let us make use of Proposition \ref{prop:rieszestimate}, which we apply to the compactly supported function $\rho_H:=\rho\bC_{\{\rho\ge H\}}/M_\rho\left(A_H\right)$.\\

\fbox{Case $1-N<k<0$} Proposition \ref{prop:rieszestimate}(i) applied to $\rho_H$ gives the estimate
\begin{equation*}
\int_{B_{A_H}}|x-y|^{k}\rho(y)\,dy\le C_1 M_\rho\left(A_H\right)|x|^{k} \, , \qquad \forall x\in \RR^N\, ,
\end{equation*}
and hence, integrating against $\rho$ on $B_{A_H}^C$ and using $\rho \leq H$ on $B^C_{A_H}$,
\begin{align*}
&\int_{B_{A_H}^C}\int_{B_{A_H}}|x-y|^{k} \rho(x)\rho(y)\,dx\,dy
\le C_1 M_\rho\left(A_H\right)\int_{B_{A_H}^C}|x|^{k} \rho(x)\,dx\notag\\
&\quad= C_1 M_\rho\left(A_H\right)
\left(
\int_{B_{A_H}^C\cap B_{H^{-q}}}|x|^{k} \rho(x)\,dx
+\int_{B_{A_H}^C\setminus {B_{H^{-q}}}} |x|^{k} \rho(x)\,dx
\right)\notag\\
&\quad\le C_1 M_\rho\left(A_H\right)
\left(
H \int_{B_{A_H}^C\cap B_{H^{-q}}}|x|^k\, dx
+H^{-kq}\int_{B_{A_H}^C\setminus {B_{H^{-q}}}} \rho(x)\,dx
\right)\notag\\
&\quad\le C_1 M_\rho\left(A_H\right)
\left(
H \sigma_N \int_{A_H}^{H^{-q}}r^{k+N-1}\,dr
+H^{-kq}
\right)\notag\\
&\quad\le C_1 M_\rho\left(A_H\right)
\left(
\frac{\sigma_N}{k+N}H^{1-q(k+N)}
+H^{-kq}
\right)\, ,\notag
\end{align*}
which conludes the proof in that case. \\

\fbox{Case $-N<k\leq 1-N$} In this case, we obtain from Proposition \ref{prop:rieszestimate}(ii) applied to $\rho_H$ the estimate
\begin{equation*}
\int_{B_{A_H}}|x-y|^{k}\rho(y)\,dy\le C_2 M_\rho\left(A_H\right)T_k(|x|,A_H)|x|^{k} \, , \qquad \forall x\in B_{A_H}^C\, ,
\end{equation*}
and integrating against $\rho(x)$ over $B_{A_H}^C$, we have
\begin{equation}\label{start}
\int_{B_{A_H}^C}\int_{B_{A_H}}|x-y|^{k} \rho(x)\rho(y)\,dx\,dy\le C_2 M_\rho\left(A_H\right)\int_{B_{A_H}^C}T_k(|x|,A_H)|x|^{k} \rho(x)\,dx\, .
\end{equation}
We split the integral in the right hand side as $I_1+I_2$, where
\[
I_1:=\int_{B_{A_H}^C\cap B_{H^{-q}}}T_k(|x|,A_H)|x|^{k} \rho(x)\,dx,\quad
I_2:=\int_{B_{A_H}^C\setminus {B_{H^{-q}}}}T_k(|x|,A_H)|x|^{k} \rho(x)\,dx\, .
\]
Let us first consider $I_2$, where we have  $|x|\ge H^{-q}\ge 2A_H$ on the integration domain. Since the map $|x|\mapsto \frac{|x|+A_H}{|x|-A_H}$ is
monotonically decreasing to $1$ in $(A_H,+\infty)$,  it is bounded above by $3$ on $(2A_H,+\infty)$. We conclude from \eqref{Pik} that
$T_k(|x|,A_H)\le3$ for $|x|\in (H^{-q},+\infty)$. This entails
\begin{equation}\label{II-estimate}
I_2\le 3\int_{B_{A_H}^C\setminus {B_{H^{-q}}}}|x|^{k}\rho(x)\,dx\le 3\, H^{-kq},
\end{equation}
where we used once again $|x|\ge H^{-q}$, recalling that $k<0$.

Concerning $I_1$, we have $\rho\le H$ on $B_{A_H}^C$ which entails
\begin{equation}\label{I-estimate}
I_1\le H \int_{B_{A_H}^C\cap B_{H^{-q}}}T_k(|x|,A_H)|x|^{k} \,dx=\sigma_N
H\int_{A_H}^{H^{-q}}T_k(r,A_H) r^{k+N-1}\,dr.
\end{equation}
If $-N<k<1-N$, we use \eqref{Pik} and $(r+2A_H)/(r+A_H)< 2$ for $r\in (0,+\infty)$, so that
\begin{equation}\label{s2}
\int_{A_H}^{H^{-q}}T_k(r,A_H) r^{k+N-1}\,dr\le\int_{0}^{H^{-q}}\left(\frac{r+2A_H}{r+A_H}\right)^{1-k-N}
\,r^{k+N-1}\,dr\le\frac{2^{1-k-N}}{k+N}\,H^{-q(k+N)}.
\end{equation}
If $k=1-N$ we have from \eqref{Pik}, since $2A_H\le H^{-q}<1$,
\begin{equation}\label{s3}\begin{aligned}
\int_{A_H}^{H^{-q}}T_k(r,A_H) r^{k+N-1}\,dr &=\int_{A_H}^{H^{-q}} \left(1+\log\left(\frac{r+A_H}{r-A_H}\right)\right)\,dr\\&\le
\int_{0}^{H^{-q}}\left(1+\log\left(\frac{r+1}{r}\right)\right)\,dr\\&=H^{-q}+H^{-q}\log(1+H^q)+\log(1+H^{-q})\\&\le H^{-q}(2+\log(1+H^q)).
\end{aligned}\end{equation}
Combining \eqref{I-estimate}, \eqref{s2}, \eqref{s3} we conclude $I_1\le
\tfrac{\sigma_N2^{1-k+N}}{k+N}\, H^{1-q(k+N)}$ if $-N<k<1-N$, and
$I_1\le \sigma_N H^{1-q}(2+\log(1+H^q))$ if $k=1-N$. These information together with the estimate \eqref{II-estimate} can be inserted into \eqref{start} to conclude.
\end{proof}

We are now in a position to prove that any minimiser of $\mF$ is bounded.

\begin{theorem}\label{prop:Linfty}
Let $\chi>0$, $k \in (-N,0)$ and $m >m_c$. Then any global minimiser of $\mF$ over $\mY$ is uniformly bounded and compactly supported.
\end{theorem}

\begin{proof}
Since $\rho$ is radially symmetric decreasing by Proposition \ref{prop:characmin}, it is enough to show $\rho(0)<\infty$.
Let us reason by contradiction and assume that $\rho$ is unbounded at the origin.
We will  show that $\mF[\rho]-\mF[\tilde \rho]>0$ for a suitably chosen competitor $\tilde \rho$,
\begin{equation*}
 \tilde \rho(x)=\tilde\rho_{H,r}(x): = \frac{N M_\rho({A_H})}{\sigma_N r^N} \bC_{D_r}(x) + \rho(x) \bC_{B_{A_H}^C}(x)\, ,
\end{equation*}
where $B_{A_H}$ and $q$ are defined as in Lemma \ref{Newtonlemma}, $B_{A_H}^C$ denotes the complement of $B_{A_H}$ and $\bC_{D_r}$ is the characteristic function of a
ball $D_r:=B_r(x_0)$ of radius $r>0$, centered at some $x_0\neq 0$ and such that $D_r\cap B_{A_H}=\emptyset$. Note that $A_H\leq H^{-q}/2< H_0^{-q}/2< 1/2$. Hence, we
can take $r>1$ and $D_r$ centered at the point $x_0=(2r,0, \dots, 0) \in \RR^N$.
 Notice in particular that since $\rho$ is unbounded, for any $H>0$ we have that $B_{A_H}$ has non-empty interior. On the other hand, $B_{A_H}$ shrinks to the origin as $H\to\infty$ since $\rho$ is integrable.

 As $D_r\subset B_{A_H}^C$ and $\rho=\tilde\rho$ on $B_{A_H}^C\setminus D_r$, we obtain
\begin{equation*}\begin{aligned}
 (m-1)\left(\mH_m[\rho]-\mH_m[\tilde \rho]\right)
 &=\int_{B_{A_H}}\rho^m\, dx + \int_{B_{A_H}^C} \rho^m\, dx - \int_{B_{A_H}^C} \left(\rho+\frac{N M_\rho({A_H})}{\sigma_N r^N}\bC_{D_r}\right)^m\, dx\\
 &=\int_{B_{A_H}}\rho^m\, dx + \int_{D_r} \left[\rho^m - \left(\rho+\frac{N M_\rho({A_H})}{\sigma_Nr^N}\right)^m\right]\, dx\,.
\end{aligned}\end{equation*}
We bound
\[\begin{aligned}
 \eps_r:
 &=\int_{D_r} \left[\rho^m - \left(\rho+\frac{N M_\rho({A_H})}{\sigma_Nr^N}\right)^m\right]\, dx
 \le
 M_\rho(A_H)^m\left(\frac{\sigma_N}{N}\right)^{1-m} r^{N(1-m)}\, ,
 \end{aligned}
\]
where we use the convexity identity $(a+b)^m \geq \left|a^m-b^m\right|$ for $a,b >0$. Hence, $\eps_r$ goes to $0$ as $r\to\infty$.
Summarising we have for any $r>1$,
\begin{equation}\label{Hest}
(m-1)\left(\mH_m[\rho]-\mH_m[\tilde \rho]\right)
 =\int_{B_{A_H}}\rho^m\, dx+\eps_r,
\end{equation}
with $\eps_r$ vanishing as $r\to\infty$.

To estimate the interaction term, we split the double integral into three parts:
\begin{equation}\label{Wdiff}\begin{aligned}
 2k\left(\mW_k[\rho]-\mW_k[\tilde \rho]\right)
 = &\iint_{\RR^N \times \RR^N} |x-y|^{k}\left(\rho(x)\rho(y)-\tilde\rho(x)\tilde\rho(y)\right)\,dxdy\\
 =&\iint_{B_{A_H} \times B_{A_H}} |x-y|^{k}\rho(x)\rho(y)\,dxdy\\
 &+2\iint_{B_{A_H} \times B_{A_H}^C} |x-y|^{k}\rho(x)\rho(y)\,dxdy\\
 &+ \iint_{B_{A_H}^C \times B_{A_H}^C} |x-y|^{k}\left(\rho(x)\rho(y)
- \tilde\rho(x)\tilde\rho(y)\right)\,dxdy
 \\
 =:&\,I_1+I_2+I_3(r)\, .
\end{aligned}\end{equation}

 Let us start with $I_3$. By noticing once again that  $\rho=\tilde\rho$ on $B_{A_H}^C\setminus D_r$ for any $r>0$, we have
\begin{align*}
 I_3(r)=&\int\int_{D_r\times D_r}|x-y|^{k}\left(\rho(x)\rho(y)-\tilde\rho(x)\tilde\rho(y)\right)\, dxdy\\
 &+2\int\int_{D_r\times (B_{A_H}^C\setminus D_r)}|x-y|^{k}\left(\rho(x)\rho(y)-\tilde\rho(x)\tilde\rho(y)\right)\, dxdy\\
 =:&\, I_{31}(r)+I_{32}(r)\, .
\end{align*}
Since $\tilde\rho=\rho+\frac{N M_\rho({A_H})}{\sigma_N r^N}$ on $D_r$, we have
\[
I_{32}(r)=-2\frac{N M_\rho({A_H})}{\sigma_Nr^N}\int\int_{D_r\times (B_{A_H}^C\setminus D_r)}|x-y|^{k}\rho(y)\,dxdy\,.
\]
By the HLS inequality \eqref{eq:HLS}, we have
\begin{align*}
 |I_{32}(r)|\le &2\frac{N M_\rho({A_H})}{\sigma_Nr^N}\int\int_{D_r\times \RR^N}|x-y|^{k}\rho(y)\,dxdy\\
\leq
&2C_{HLS}\frac{N M_\rho({A_H})}{\sigma_Nr^N}
\|\bC_{D_r}\|_a\|\rho\|_{b}
\end{align*}
if $a>1, b>1$ and $1/a+1/b-k/N=2$.  We can choose $b\in \left(1,\min\left\{m\, , \,N/(k+N)\right\}\right)$, which is possible as $-N<k<0, m>1$, and then we
get $a>1$, $\rho\in L^b(\mathbb{R}^N)$ as $1<b<m$, and
\[
|I_{32}(r)|\le2C_{HLS}||\rho||_b M_\rho({A_H})\,\left(\frac{\sigma_N r^N}{N}\right)^{\frac1a-1}\, .
\]
The latter vanishes as $r\to\infty$. For the term $I_{31}$, we have
\begin{align*}
 I_{31}(r)=
 &-2\frac{NM_\rho({A_H})}{\sigma_Nr^N}\int\int_{D_r\times D_r}|x-y|^{k}\rho(y)\,dxdy\\
&-\left(\frac{NM_\rho({A_H})}{\sigma_Nr^N}\right)^2\int\int_{D_r\times D_r}|x-y|^{k}\,dxdy\, .
\end{align*}
With the same choice of $a,b$ as above, the HLS inequality implies
\begin{align*}
 |I_{31}(r)|\leq
&2\frac{N M_\rho({A_H})}{\sigma_Nr^N}\int\int_{D_r\times \RR^N}|x-y|^{k}\rho(y)\,dxdy\\
&+
\left(\frac{NM_\rho({A_H})}{\sigma_Nr^N}\right)^2\int\int_{D_r\times D_r}|x-y|^{k}\,dxdy\\
\le &C_{HLS}M_\rho({A_H})\,\left(2||\rho||_b\left(\frac{\sigma_N r^N}{N}\right)^{\frac1a-1}+ M_\rho({A_H}) \left(\frac{\sigma_N
r^N}{N}\right)^{\frac1a+\frac1b-2}\right)\, ,
\end{align*}
which vanishes as $r\to\infty$ since $a>1$ and $b>1$.
We conclude that $I_3(r)\to 0$ as $r\to\infty$.

The  integral $I_1$ can be estimated using Theorem \ref{thm:HLSm}, and the fact that $\rho\ge H>1$ on $B_{A_H}$ together with $m>m_c$,
\begin{align}
 I_1= \iint_{B_{A_H} \times B_{A_H}} |x-y|^{k}\rho(x)\rho(y)\,dxdy
 &\leq C_* M_\rho(A_H)^{1+k/N} \int_{B_{A_H}}\rho^{m_c}(x)\, dx\notag\\
  &\leq C_* M_\rho(A_H)^{1+k/N} \int_{B_{A_H}}\rho^{m}(x)\, dx\, .\label{I1estimate}
\end{align}

On the other hand, the HLS inequalities \eqref{eq:HLS} and \eqref{eq:HLSm} do not seem to give a sharp enough estimate for the cross-term $I_2$, for which we instead invoke Lemma
\ref{Newtonlemma}, yielding
\begin{equation}\label{I2estimate}
I_2\le 2C_{k,N}\, M_{\rho}(A_H)\, \mathcal{K}_{k,q,N}(H),
\end{equation}
for given  $q\in [0,m/N)$ and large enough $H$ as specified in Lemma \ref{Newtonlemma}.

In order to conclude, we join together \eqref{Hest}, \eqref{Wdiff}, \eqref{I1estimate}  and \eqref{I2estimate} to obtain for any $r>1$ and  any large enough $H$,
\begin{align}
 \mF[\rho]-\mF[\tilde \rho]&= \mH_m[\rho]-\mH_m[\tilde\rho]+\chi\left(\mW_k[\rho]-\mW_k[\tilde\rho]\right) \notag\\
 &\ge \left(\frac{1}{m-1} +\chi\frac{C_*}{2k}  M_\rho({A_H})^{1+k/N}\right) \,\int_{B_{A_H}}\rho^{m}\, dx
 +\chi\,\frac{C_{k,N}}{k}\, M_{\rho}({A_H})\, \mathcal{K}_{s,q,N}(H)\notag\\
 &\quad
 +\frac{\eps_r}{m-1}
 +\frac{\chi}{2k} I_3(r)\,.\label{comparison}
\end{align}

Now we choose $q$. On the one hand, notice that for a choice $\eta>0$ small enough such that $m>m_c+\eta$, we have
\begin{equation}\label{n1}
 \frac{2-m+\eta}{k+N}<\frac{m-1-\eta}{(-k)}\, .
\end{equation}
On the other
hand, $-N<k<0$ implies $1-k/N>2N/(2N+k)$. Since $m>m_c$, this gives the inequality $m>2N/(2N+k)$. Hence, for small enough $\eta>0$ such that $m>N(2+\eta)/(2N+k)$, we have
\begin{equation}\label{n2}
\frac{2-m+\eta}{k+N}<\frac mN.
\end{equation}
Thanks to \eqref{n1} and \eqref{n2} we see that
 we can fix a non-negative $q$ such that
\begin{equation}\label{theq}
 \frac{2-m+\eta}{k+N}<q<\min\left\{\frac mN\, ,\, \frac{(m-1-\eta)}{(-k)}\right\}.
\end{equation}
Since $q$ satisfies \eqref{theq}, it follows that $-kq<m-1-\eta$ and at the same time $1-q(k+N)<m-1-\eta$, showing that $\mathcal{K}_{k,q,N}(H) $
from Lemma \ref{Newtonlemma}  grows slower than $H^{m-1-\eta}$ as $H\to \infty$ for $k \neq 1-N$. If $k=1-N$, we have that for any $C>0$
there exists $H>H_0$ large enough such that $CH^{1-q}\log(1+H^q)<H^{m-1-\eta}$ since $q>2-m+\eta$, and so the same result follows. Hence,  for any large enough $H$ we have
\[
C_{k,N} \,M_{\rho}({A_H})\, \mathcal{K}_{k,q,N}(H)< C_{k,N} H^{m-1-\eta}\, M_\rho({A_H})\le C_{k,N} H^{-\eta}\int_{B_{A_H}}\rho^{m}\, dx
\]
since $\rho\ge H$ on $B_{A_H}$. Inserting the last two estimates in \eqref{comparison} we get for some $\eta>0$
\begin{align*}
 \mF[\rho]-\mF[\tilde \rho]
 &\ge \left(\frac{1}{m-1} +\chi\frac{C_*}{2k}  M_\rho({A_H})^{1+k/N} +  \chi\,\frac{C_{k,N} H^{-\eta}}{k}\right) \,\int_{B_{A_H}}\rho^{m}\, dx
 \notag\\
 &\quad
 +\frac{\eps_r}{m-1}
 +\frac{\chi}{2k} I_3(r)\,.
\end{align*}
for any  $r>1$ and any large enough $H$. First of all, notice that $\int_{B_{A_H}}\rho^{m}\, dx$ is strictly positive since we are assuming that $\rho$ is unbounded. We can therefore fix $H$ large enough such that the constant in front of $\int_{B_{A_H}}\rho^{m}$ is strictly positive. Secondly, we have already proven that $\eps_r$ and $I_3(r)$ vanish as $r\to \infty$, so we can choose $r$ large enough such that
\[
 \mF[\rho]-\mF[\tilde \rho]>0\, ,
\]
contradicting the minimality of $\rho$. We conclude that minimisers of $\mF$ are bounded. Finally, we can just use the Euler--Lagrange equation \eqref{ELmin} and the same argument as for Corollary \ref{cor:compactsupp} to prove that $\rho$ is compactly supported.
\end{proof}
\subsection{Regularity Properties of Global Minimisers}\label{sec:regmins}

This section is devoted to the regularity properties of global minimisers. With enough regularity, global minimisers satisfy the conditions of Definition \ref{def:sstates}, and are therefore stationary states of equation \eqref{eq:KS}. This will allow us to complete the proof of Theorem \ref{thm:main1}.

We begin by introducing some notation and preliminary results. As we will make use of the H\"{o}lder regularising properties of the fractional Laplacian, see \cite{Ros-Oton,Silvestre}, the notation
$$
c_{N,s}(-\Delta)^s S_k=\rho\,, \quad s \in (0,N/2)
$$
is better adapted to the arguments that follow, fixing $s=(k+N)/2$, and we will therefore state the results in this section in terms of $s$.\\
One fractional regularity result that we will use repeatedly in this section follows directly from the HLS inequality \eqref{eq:HLS} applied with $k=2s-N$: for any 
$$
 s\in (0,N/2)\, , \qquad  1<p<\frac{N}{2s}\, ,
 \qquad q=\frac{Np}{N-2sp}\, ,
$$
we have
\begin{equation}\label{fracreg}
 (-\Delta)^s f \in L^p\left(\RR^N\right)
 \; \Rightarrow \;
 f \in L^q\left(\RR^N\right)\, .
\end{equation}

Further, for $1\leq p<\infty$ and $s\geq 0$, we define the \emph{Bessel potential space} $\mathcal{L}^{2s,p}(\RR^{N})$
 as made by all functions $f\in L^{p}(\RR^{N})$ such that
$
(I-\Delta)^{s}f\in L^{p}(\RR^{N}),
$ meaning that $f$ is the Bessel potential of an $L^{p}(\mathbb{R}^N)$ function (see \cite[pag. 135]{Stein}). Since we are working with the operator $(-\Delta)^s$ instead of $(I-\Delta)^s$,
we make use of a characterisation of the space $\mathcal{L}^{2s,p}(\mathbb{R}^N)$ in terms of Riesz potentials.  For $1<p<\infty$ and $0<s<1$ we have
\begin{equation}\label{bessel2}
\mathcal{L}^{2s,p}(\RR^{N})=\left\{f\in L^{p}(\RR^{N}): f=g\ast W_{2s-N}, 	\; g\in L^{p}(\RR^{N})\right\},
\end{equation}
see \cite[Theorem 26.8, Theorem 27.3]{Samko}, see also exercise 6.10 in Stein's book \cite[pag. 161]{Stein}.
Moreover, for $1\leq p<\infty$ and $0<s<1/2$
we define the \emph{fractional Sobolev} space $\mW^{2s,p}(\RR^{N})$ by
\[
\mW^{2s,p}\left(\RR^{N}\right):= \left\{f \in L^{p}(\RR^{N}) :
\iint_{\RR^{N}\times\RR^{N}}\frac{|f(x)-f(y)|^{p}}{|x-y|^{N+2s p}}dx\,dy
<\infty
\right\}.
\]
We have the embeddings
\begin{equation}\label{besovembedding}
\mathcal{L}^{2s,p}(\mathbb{R}^N)\subset \mW^{2s,p}(\mathbb{R}^N)\quad \mbox{for}\quad p\geq 2 ,\quad s\in(0,1/2)\, ,
\end{equation}
\begin{equation}\label{sobolevembedding}
\mW^{2s,p}\left(\RR^N\right) \subset C^{0,\beta}\left(\RR^N\right)\quad \mbox{for}\quad \beta= 2s-N/p,\quad p> N/2s ,\quad s\in(0,1/2),
\end{equation}
  see   \cite[pag. 155]{Stein} and \cite[Theorem 4.4.7]{Demengel} respectively.

Let $s\in (0,1)$ and $\alpha > 0$ such that $\alpha+2s$ is not an integer. Since $c_{N,s}(-\Delta)^sS_k=\rho$ holds in $\mathbb{R}^N$, then we have from \cite[Theorem 1.1, Corollary 3.5]{Ros-Oton} (see also \cite[Proposition 5.2]{CafSting}) that
\begin{equation}\label{ros}
\|S_k\|_{C^{0,\alpha+2s}(\overline{B_{1/2}(0)})}\le c\left(\|S_k\|_{L^{\infty}(\mathbb{R}^N)}+\|\rho\|_{C^{0,\alpha}(\overline{B_1(0)})}\right)\,,
\end{equation}
  with the convention that if $\alpha\geq 1$ for any open set $U$ in $\RR^{N}$, then $C^{0,\alpha}(\overline{U}):=C^{\alpha',\alpha''}(\overline{U})$, where $\alpha'+\alpha''=\alpha$, $\alpha''\in[0,1)$ and $\alpha'$ is the greatest integer less than or equal to $\alpha$. With this notation, we have $C^{0,1}(\RR^N)=C^{1,0}(\RR^N)=\mW^{1,\infty}(\RR^N)$. In particular, using \eqref{ros} it follows that for $\alpha > 0$, $s \in (0,1)$ and $\alpha+2s$ not an integer,
  \begin{equation}\label{HR}
  \|S_k\|_{C^{0,\alpha+2s}(\RR^{N})}\le c\left(\|S_k\|_{L^{\infty}(\mathbb{R}^N)}+\|\rho\|_{C^{0,\alpha}(\RR^{N})}\right)\, .
  \end{equation}
Moreover, rescaling inequality \eqref{ros} in any ball $B_{R}(x_0)$ where $R\neq1$ we have the estimate
\begin{equation}\begin{aligned}
&\sum_{\ell=0}^{\alpha_{2}}R^{\ell}\|D^{\ell} S_{k}\|_{L^{\infty}(B_{R/2}(x_{0}))}+R^{\alpha+2s}[D^{\alpha_{1}}S_{k}]_{C^{0,\alpha+2s-\alpha_{2}}(B_{R/2}(x_{0}))}
\\
&\leq C\left[\|S_{k}\|_{L^{\infty}(\mathbb{R}^N)}+ \sum_{\ell=0}^{\alpha_{1}}R^{2s+\ell}\|D^{\ell} \rho\|_{L^{\infty}(B_{R}(x_{0}))}
+R^{\alpha+2s}[D^{\alpha_{1}}\rho]_{C^{0,\alpha-\alpha_{1}}(B_{R}(x_{0}))}  \right]\label{rosrescaled}
\end{aligned}\end{equation}
where $\alpha_{1}, \alpha_{2}$ are the greatest integers less than $\alpha$ and $\alpha+2s$ respectively. In \eqref{rosrescaled} the quantities $\|D^{\ell} S_{k}\|_{L^{\infty}}$ and $[D^{\ell}\rho]_{C^{0,\alpha}}$ denote the sum of the $L^{\infty}$ norms and the $C^{0,\alpha}$ seminorms of the derivatives $D^{(\beta)} S_{k}$, $D^{(\beta)} \rho$ of order $\ell$ (that is $|\beta|=\ell$).


Finally, we recall the definition of $m_c$ and $m^*$ in \eqref{mstar} in terms of $s$:
$m_c:=2-\frac{2s}{N}$ and
 \begin{equation*}
  m^*:=
  \begin{cases}
  \dfrac{2-2s}{1-2s}\,  \qquad &\text{if} \quad N\geq 1 \quad \text{and} \quad s \in (0,1/2)\, , \\
  + \, \infty &\text{if} \quad N\geq 2 \quad \text{and} \quad s \in [1/2,N/2)\, .
  \end{cases}
 \end{equation*}

Let us begin by showing that global minimisers of $\mF$ enjoy the good H\"{o}lder regularity in the most singular range, as long as diffusion is not too slow.

\begin{theorem}\label{thm:regmin}
 Let $\chi>0$ and $s \in (0,N/2)$. If $m_c<m< m^*$, then any global minimiser $\rho \in \mY$ of $\mF$ satisfies $S_k=W_k\ast \rho \in\mathcal{W}^{1,\infty}(\mathbb{R}^N)$, $\rho^{m-1}\in\mathcal{W}^{1,\infty}(\mathbb{R}^N)$ and $\rho \in C^{0,\alpha}(\RR^N)$ with $\alpha=\min\{1,\tfrac{1}{m-1}\}$.
\end{theorem}

\begin{proof}
 Recall that the global minimiser $\rho \in \mY$ of $\mF$ is radially symmetric non-increasing and compactly supported by Theorem \ref{thm:Emins} and Theorem \ref{prop:Linfty}. Since $\rho \in L^1\left(\RR^N\right)\cap L^\infty\left(\RR^N\right)$ by Theorem \ref{prop:Linfty}, we have $\rho \in L^p\left(\RR^N\right)$ for any $1< p<\infty$. Since $\rho=c_{N,s}(-\Delta)^s S_k$, it follows from \eqref{fracreg} that $S_k\in L^{q}(\mathbb{R}^N)$, $q=\tfrac{Np}{N-2sp}$ for all $1<p<\tfrac{N}{2s}$, that is $S_k\in L^{p}(\mathbb{R}^N)$ for all $p \in (\tfrac{N}{N-2s},\infty)$.
   Then, if $s\in(0,1)$, since $S_{k}$ is the Riesz potential of the density $\rho$ in $L^{p}$,   
   by the characterisation \eqref{bessel2} of the Bessel potential space, we conclude that $S_k \in \mL^{2s,p}(\RR^N)$ for all $p>\tfrac{N}{N-2s}$. Let us first consider $s<1/2$, as the cases $1/2<s<N/2$ and $s=1/2$ follow as a corollary. \\

 \fbox{$0<s<1/2$} In this case, we have the embedding \eqref{besovembedding} and so $S_k \in \mW^{2s,p}(\RR^N)$ for all $p\geq 2>\tfrac{N}{N-2s}$ if $N\geq 2$ and for all  $p>\max\{2,\tfrac{1}{1-2s}\}$ if $N=1$. Using \eqref{sobolevembedding}, we conclude that $S_k \in C^{0,\beta}\left(\RR^N\right)$  with   $$\beta: = 2s-N/p,$$ for any $p>\tfrac{N}{2s}> 2$   if $N\geq2$ and for any $p>\max\{\tfrac{1}{2s},\tfrac{1}{1-2s}\}$ if $N=1$. Hence $\rho^{m-1} \in C^{0,\beta}\left(\RR^N\right)$ for the same choice of $\beta$ using the Euler--Lagrange condition \eqref{ELmin} since $\rho^{m-1}$ is the truncation of a function which is $S_k$ up to a constant.

	  Note that $m_c \in (1,2)$ and $m^*>2$. In what follows we split our analysis into the cases $m_c<m\leq 2$ and $2<m < m^*$, still assuming $s<1/2$.
 If $m\leq 2$, the argument follows along the lines of \cite[Corollary 3.12]{CCH1} since $\rho^{m-1} \in C^{0,\alpha}(\RR^N)$ implies that $\rho$ is in the same H\"{o}lder space for any $\alpha \in (0,1)$.
    Indeed, in such case we bootstrap in the following way.
  Let us fix $n \in \N$ such that
 \begin{equation}\label{bootrange}
 \frac{1}{n+1}<2s\leq \frac{1}{n}
 \end{equation}
 and let us define
 \begin{equation}\label{bn}
 \beta_n:=\beta+(n-1)2s=2ns-N/p.
 \end{equation}
 Form \eqref{bootrange} and \eqref{bn} we see that by choosing large enough $p$ there hold $1-2s<\beta_n<1$.
  Note that $S_k \in L^\infty\left(\RR^N\right)$ by Lemma \ref{lem:regS}, and  if   $\rho \in C^{0,\gamma}\left(\RR^N\right)$ for some $\gamma \in (0,1)$ such that $\gamma+2s<1$, then $S_k \in C^{0,\gamma+2s}\left(\RR^N\right)$ by \eqref{HR}, implying $\rho^{m-1} \in C^{0,\gamma+2s}\left(\RR^N\right)$ using the Euler--Lagrange conditions \eqref{ELmin}. Therefore $\rho \in C^{0,\gamma+2s}\left(\RR^N\right)$ since $m \in (m_c,2]$. Iterating this argument $(n-1)$ times starting with $\gamma=\beta$ gives $\rho \in C^{0,\beta_n}\left(\RR^N\right)$  . Since $\beta_n<1$ and $\beta_n+2s>1$, a last application of \eqref{HR} yields $S_k\in\mathcal{W}^{1,\infty}(\mathbb{R}^N)$, so that $\rho^{m-1}\in\mathcal{W}^{1,\infty}(\mathbb{R}^N)$, thus $\rho\in\mathcal{W}^{1,\infty}(\mathbb{R}^N)$.
 This concludes the proof in the case $m\leq 2$.

 Now, let us assume $2<m < m^*$ and $s<1/2$. Recall that $\rho^{m-1}\in C^{0,\gamma}\left(\RR^N\right)$ for any $\gamma<2s$, and so $\rho\in C^{0,\gamma}\left(\RR^N\right)$ for any $\gamma<\tfrac{2s}{m-1}$. By \eqref{HR} we get $S_k\in C^{0,\gamma}\left(\RR^N\right)$ for any $\gamma<\tfrac{2s}{m-1}+2s$, and the same for $\rho^{m-1}$ by the Euler--Lagrange equation \eqref{ELmin}.   Once more with a bootstrap  argument, we obtain improved H\"{o}lder regularity for $\rho^{m-1}$. Indeed, since
\begin{equation}\label{series}
\sum_{j=0}^{+\infty}\frac{2s}{(m-1)^j}=\frac{2s(m-1)}{m-2}\,
\end{equation}
and since  $m < m^*$ means $\tfrac{2s(m-1)}{m-2} > 1$, after taking a suitably large number of iterations we get $S_k\in \mathcal{W}^{1,\infty}(\mathbb{R}^N)$ and $\rho^{m-1} \in \mathcal{W}^{1,\infty}(\RR^N)$.  Hence, $\rho \in  C^{0,{1}/{(m-1)}}\left(\RR^N\right)$.\\

\fbox{$N\geq 2$, $1/2\leq s<N/2$}
   We start with the case $s=1/2$. We have $S_k\in L^p(\mathbb{R}^N)$ for any $p>\tfrac{N}{N-1}$ as shown at the beginning of this proof. By \eqref{bessel2} we get $S_k \in \mL^{1,p}\left(\RR^N\right)$ for all $p>\tfrac{N}{N-1}$. Then we also have $S_k \in \mL^{2r,p}(\RR^N)$ for all $p>\tfrac{N}{N-1}$ and  for all $r \in (0,1/2)$ by the embeddings between Bessel potential spaces, see \cite[pag. 135]{Stein}.
 Noting that $2\geq \tfrac{N}{N-1}$ for $N\geq 2$, by \eqref{besovembedding} and \eqref{sobolevembedding} we get $S_k\in C^{0,2r-N/p}(\mathbb{R}^N)$ for any $r\in(0,1/2)$ and any $p>\tfrac{N}{2r}$. That is, $S_k\in C^{0,\gamma}(\mathbb{R}^N)$ for any $\gamma\in (0,1)$. 
 By the Euler--Lagrange equation \eqref{ELmin}, $\rho \in C^{0,\gamma \alpha}(\RR^N)$ with $\alpha=\min\{1,\tfrac{1}{m-1}\}$, and so \eqref{HR} for $s=1/2$ implies $S_k\in\mathcal{W}^{1,\infty}(\mathbb{R}^N)$. Again by the Euler--Lagrange equation \eqref{ELmin}, we obtain $\rho^{m-1}\in\mathcal{W}^{1,\infty}(\mathbb{R}^N)$.\\
If $1/2<s<N/2$ on the other hand, we obtain directly that $S_k \in \mW^{1,\infty}(\RR^N)$ by Lemma \ref{lem:regS}, and so 	    $\rho^{m-1}\in \mathcal{W}^{1,\infty}(\mathbb{R}^N)$.  \\
We conclude that $\rho\in C^{0,\alpha}(\mathbb{R}^N)$ with $\alpha=\min\{1,\tfrac{1}{m-1}\}$ for any $1/2\leq s<N/2$.
\end{proof}

\begin{remark}\label{singularrange} \rm
 If $m\geq m^*$ and $s<1/2$, we recover some H\"{o}lder regularity, but it is not enough to show that global minimisers of $\mF$ are stationary states of \eqref{eq:KS}. More precisely, $m \geq m^*$ means $\tfrac{2s(m-1)}{m-2} \leq 1 $, and so it follows from \eqref{series} that $\rho\in C^{0,\gamma}\left(\RR^N\right)$ for any $\gamma<\tfrac{2s}{m-2}$. Note that $m \geq m^*$ also implies $\tfrac{2s}{m-2}\leq 1-2s$, and we are therefore not able to go above the desired H\"{o}lder exponent $1-2s$.
\end{remark}

\begin{remark}\rm
In the arguments of Theorem \ref{thm:regmin} one could choose to
directly bootstrap on fractional Sobolev spaces.
In fact, for $0<s<1/2$  and $m > 2$  we have that $\rho^{m-1}\in \mW^{2s,p}(\RR^N)$ implies $\rho\in \mW^{\frac{2s}{m-1},p(m-1)}(\RR^N)$.
Indeed, let $\alpha<1$ and $u\in \mW^{\alpha,p}(\RR^N)$, where  and $p\in [1,\infty)$. By the algebraic inequality $||a|^{\alpha}-|b|^{\alpha}|\leq C|a-b|^{\alpha}$ we have
\[
\iint_{\RR^N\times\RR^N}\frac{||u(x)|^{\alpha}-|u(y)|^{\alpha}|^{p/\alpha}}{|x-y|^{N+\alpha 2s(p/\alpha)}}\,dxdy\leq c \iint_{\RR^N\times\RR^N}\frac{|u(x)-u(y)|^{p}}{|x-y|^{N+2sp}}\,dxdy\, ,
\]
thus $|u|^{\alpha}\in \mW^{\alpha s,p/\alpha}(\RR^N)$.  This property is also valid for Sobolev spaces with integer order, see \cite{Mironescu}.
In particular, thanks to this property, in case $m\ge m^*$ we may obtain $\rho^{m-1}\in \mW^{\alpha,p}(\RR^N)$ for any $\alpha<\frac{2s(m-1)}{m-2}$ and any large enough $p$, hence \eqref{sobolevembedding} implies that $\rho$ has the H\"{o}lder regularity stated in Remark \ref{singularrange}.
\end{remark}

We are now ready to show that global minimisers possess the good regularity properties to be stationary states of equation \eqref{eq:KS} according to Definition \ref{def:sstates}.

\begin{theorem}\label{thm:minsstates}
 Let $\chi>0$, $s \in (0,N/2)$ and $m_c<m < m^*$. Then all global minimisers of $\mF$ in $\mY$ are stationary states of equation \eqref{eq:KS} according to Definition \ref{def:sstates}.
\end{theorem}

\begin{proof}
Note that $m<m^*$ means $1-2s<1/(m-1)$, and so thanks to Theorem \ref{thm:regmin}, $S_k$ and $\rho$ satisfy the regularity conditions of Definition \ref{def:sstates}. Further, since $\rho^{m-1} \in \mathcal{W}^{1,\infty}\left(\RR^N\right)$, we can take gradients on both sides of the Euler--Lagrange condition \eqref{ELmin}. Multiplying by $\rho$ and writing $\rho \nabla \rho^{m-1}=\tfrac{m-1}{m}\nabla \rho^{m}$, we conclude that global minimisers of $\mF$ in $\mY$ satisfy relation \eqref{eq:sstates} for stationary states of equation \eqref{eq:KS}.
\end{proof}

In fact, we can show that global minimisers have even more regularity inside their support.

\begin{theorem}\label{cor:Cinfty}
 Let $\chi>0$,  $m_c<m$ and $s \in (0,N/2)$. If $\rho \in \mY$ is a global minimiser of $\mF$, then $\rho$ is $C^\infty$ in the interior of its support.
\end{theorem}

\begin{proof}
 By Theorem \ref{thm:regmin} and Remark \ref{singularrange}, we have $\rho \in C^{0,\alpha}(\RR^N)$ for some $\alpha \in (0,1)$.
Since $\rho$ is radially symmetric non-increasing, the interior of $\supp(\rho)$ is a ball centered at the origin, which we denote by $B$. Note also that $\rho \in L^1(\RR^N) \cap L^\infty(\RR^N)$ by Theorem \ref{prop:Linfty}, and so $S_k \in L^\infty(\RR^N)$ by Lemma \ref{lem:regS}.\\
 Assume first that $s\in (0,1)\cap (0,N/2)$.
 Applying \eqref{rosrescaled} with $B_R$ centered at a point within $B$ and such that $B_R\subset\subset B$, we obtain   $S_k\in C^{0,\gamma}(B_{R/2})$ for any $\gamma<\alpha+2s$.   It follows from the Euler--Lagrange condition \eqref{ELmin} that $\rho^{m-1}$ has the same regularity as $S_k$ on $B_{R/2}$, and since $\rho$ is bounded away from zero on $B_{R/2}$, we conclude $\rho \in C^{0,\gamma}(B_{R/2})$ for any $\gamma<\alpha+2s$.
 Repeating the previous step now on $B_{R/2}$, we get the improved regularity   $S_k\in C^{0,\gamma}(B_{R/4})$ for any $\gamma<\alpha+4s$   by \eqref{rosrescaled}, which we can again transfer onto $\rho$ using \eqref{ELmin}, obtaining   $\rho\in C^{0,\gamma}(B_{R/4})$ for any $\gamma<\alpha+4s$.
Iterating, any order $\ell$ of differentiability for $S_{k}$ (and then for $\rho$) can be reached  in a neighborhood of the center of $B_R$. We notice that the argument can be applied starting from any point $x_{0}\in B$, and hence $\rho\in C^{\infty}(B)$.\\
   When $N\geq 3$ and $s \in [1,N/2)$, we take numbers $s_1,\ldots, s_l$ such that $s_i\in (0,1)$ for any $i=1,\ldots, l$ and such that $\sum_{i=1}^l s_i=s$. We also let
   $$
   S_k^{l+1}:=S_k\, , \quad S_k^j:=\Pi_{i=j}^l (-\Delta)^{s_j} S_k\, , \qquad \forall \, j \in \{1, \dots, l\}\, .
   $$
   Then $S_k^1=\rho$.
Note that Lemma \ref{lem:regS}(i) can be restated as saying that $\rho \in \mY \cap L^\infty(\RR^N)$ implies $(-\Delta)^{-\delta} \rho \in L^\infty(\RR^N)$ for all $\delta \in (0,N/2)$. Taking $\delta=s-r$ for any $r \in (0,s)$, we have $(-\Delta)^{r}S_k =(-\Delta)^{r-s}\rho \in L^\infty$.
In particular, this means $S_k^j\in L^\infty(\mathbb{R}^N)$ for any $j=1,\ldots, l+1$.
   Moreover,
   there holds
   $$
   (-\Delta)^{s_j} S_k^{j+1}=S_k^j\, , \qquad \forall \, j \in \{1, \dots, l\}\, .
   $$
    Therefore we may recursively apply \eqref{rosrescaled}, starting from $S_k^1=\rho\in C^{0,\alpha}(B_R)$, where the ball $B_R$ is centered at a point within $B$ such that $B_R\subset\subset B$, and   using the iteration rule
    \begin{align*}
     &S_k^j \in C^{0,\gamma}(B_\sigma) \; \Rightarrow \:
    S_k^{j+1} \in C^{0,\gamma+2s_j}\left(B_{\sigma/2}\right) \\
    &\forall \, j \in \{1, \dots, l\}\, , \quad \forall \,\gamma>0 \mbox{ s.t. $\gamma+2s_j$ is not an integer,}\,  \quad \forall\, B_\sigma \subset \subset B.
    \end{align*}
    We obtain   $S_{k}^{l+1}=S_k\in C^{0,\gamma}(B_{R/(2^l)})$ for any $\gamma<\alpha+2s$,   and as before, the Euler--Lagrange equation \eqref{ELmin} implies that   $\rho\in C^{0,\gamma}(B_{R/(2^l)})$ for any $\gamma<\alpha+2s$.   If we repeat the argument, we gain $2s$ in H\"{o}lder regularity for $\rho$ each time we divide the radius $R$ by $2^l$. In this way, we can reach any differentiability exponent for $\rho$ around any point of $B$, and thus $\rho\in C^{\infty}(B)$.
\end{proof}

\begin{remark}\rm We observe that the smoothness of minimisers in the interior of their support also holds in the fair competition regime $m=m_c$.
In such case global H\"{o}lder regularity was obtained in \cite{CCH1}. 
\end{remark}

The main result Theorem \ref{thm:main1} follows from Theorem \ref{thm:radiality}, Corollary \ref{cor:compactsupp}, Theorem \ref{thm:Emins}, Proposition \ref{prop:characmin}, Theorem \ref{prop:Linfty}, Theorem \ref{thm:minsstates} and Theorem \ref{cor:Cinfty}.

\section{Uniqueness}\label{sec:!diffdom}
\subsection{Optimal Transport Tools}
\label{sec:Optimal Transport Tools-diffdom}
Optimal transport is a powerful tool for reducing functional
inequalities onto pointwise inequalities. In other words, to pass from microscopic
inequalities between particle locations to macroscopic
inequalities involving densities. This sub-section summarises the main results of optimal
transportation we will need in the one-dimensional setting. They were already used in \cite{CaCa11, CCHCetraro}, where we refer for detailed proofs.

Let $\tilde \rho$ and $\rho$ be two probability densities.
According to \cite{Brenier91,McCann95}, there exists a convex
function $\psi$ whose gradient pushes forward the measure
$\tilde\rho(a) da$ onto $\rho(x) dx$: $\psi'\#
\left(\tilde\rho(a)  da\right) = \rho(x)  dx$. This convex
function satisfies the Monge-Amp\`ere equation in the weak sense:
for any test function $\varphi\in C_b(\RR)$, the following
identity holds true
\begin{equation*}
\int_{\RR} \varphi(\psi'(a)) \tilde\rho(a)\, da =
\int_{\RR} \varphi(x) \rho(x)\, dx\, . 
\end{equation*}
The convex map is unique a.e. with respect to $\rho$ and it gives
a way of interpolating measures using displacement convexity \cite{McCann97}.
On the other hand, regularity of the transport map is a
complicated matter. Here, as it was already done in \cite{CaCa11,CCHCetraro}, 
we will only use the fact that $\psi''(a) da$ can be decomposed in an absolute continuous part
$\psi_{ac}''(a)da$ and a positive singular measure \cite[Chapter 4]{Villani03}. In one dimension, the
transport map $\psi'$ is a non-decreasing function, therefore it is
differentiable a.e. and it has a countable number of jump
singularities. 
For any measurable function $U$, bounded below such
that $U(0) = 0$ we have \cite{McCann97}
\begin{equation}\label{eq:change variables-diffdom}
\int_{\RR} U(\tilde\rho(x)) \, dx = \int_{\RR}
U\left(\dfrac{\rho(a)}{\psi_{ac}''(a)}\right)\psi_{ac}''(a)\, da\, .
\end{equation}
The following Lemma proved in \cite{CaCa11} will be used to
estimate the interaction contribution in the free energy.
\begin{lemma}\label{lem:interaction-diffdom}
Let $\mathcal K:(0,\infty)\to \RR$ be an increasing and strictly
concave function. Then, for any $a, b \in \RR$
\begin{equation} \label{eq:jensen 1D}
\mathcal K\left(  \frac{\psi'(b)-\psi'(a)}{b-a} \right) \geq
\int_{0}^1 \mathcal K \left( \psi_{\AC}''([a,b]_s) \right)\, ds\,
,
\end{equation}
where the convex combination of $a$ and $b$ is given by
$[a,b]_s=(1-s)a+sb$.
Equality is achieved in \eqref{eq:jensen 1D} if and only if the
distributional derivative of the transport map $\psi''$ is a
constant function.
\end{lemma}

\subsection{Functional Inequality in One Dimension}

In what follows, we will make use of a characterisation of stationary states based on some integral reformulation of the necessary condition stated in
Proposition \ref{prop:characmin}. This characterisation was also the key idea in \cite{CaCa11,CCHCetraro} to analyse the asymptotic stability of steady
states and the functional inequalities behind.

\begin{lemma}[Characterisation of stationary states] \ \label{lem:char crit}
Let $N=1$, $\chi>0$ and $k\in (-1,0)$. If $m > m_c$ with $m_c=1-k$, then any stationary state $\bar \rho \in \mY$ of system \eqref{eq:KS} can be written in the
form
\begin{equation} \label{eq:charac sstates}
\bar \rho(p)^m = \frac{\chi}{2} \int_{\RR}\int_{0}^1 |q|^{k}\bar \rho(p -sq)\bar \rho(p-sq+q)\, dsdq\, .
\end{equation}
\end{lemma}
The proof follows the same methodology as for the fair-competition regime  \cite[Lemma 2.8]{CCHCetraro} and we omit it here.
\begin{theorem}\label{thm:HLSmequiv}
 Let $N=1$, $\chi>0$, $k \in (-1,0)$ and $m>m_c$. If \eqref{eq:KS} admits a stationary density $\bar \rho$ in $\mY$, then
 \begin{equation*} \label{fineqcritical-diffdom}
 \mF[\rho]\geq \mF[\bar \rho], \quad \forall \rho \in \mY
 \end{equation*}
 with the equality cases given by dilations of $\bar \rho$.
\end{theorem}
\begin{proof}
For a given stationary state $\bar \rho \in \mY$ and a given $\rho \in \mY$, we denote by $\psi$ the convex function whose gradient
pushes forward the measure $\bar \rho(a) da$ onto $\rho(x) dx$: $\psi' \# \left(\bar\rho(a)  da\right) = \rho(x)  dx$.
Using \eqref{eq:change variables-diffdom}, the functional $\mF[\rho]$ rewrites as follows:
\begin{align*}
\mF[\rho]& = \dfrac1{m-1}\int_\RR \left(\dfrac{\bar \rho(a)}{\psi_{ac}''(a)}\right)^{m-1} \bar \rho(a)\, da\\
&\quad+ \dfrac{\chi}{2k} \iint_{\RR\times\RR}\left|\dfrac{\psi'(a)-\psi'(b)}{a-b}\right|^{k} |a-b|^{k} \bar \rho(a) \bar \rho(b) \, da db
\\
&= \dfrac1{m-1} \int_\RR \left(\psi_{ac}''(a)\right)^{1-m} \bar \rho(a)^m\, da\\
&\quad+ \dfrac{\chi}{2k}  \iint_{\RR\times\RR} \bla \psi''([a,b]) \bra^{k} |a-b|^{k} \bar \rho(a) \bar \rho(b) \, da db   \, ,
\end{align*}
where $\bla u([a,b]) \bra = \int_0^1 u([a,b]_s)\, ds$ and $[a,b]_s=(1-s)a+bs$ for any $a,b \in \RR$ and $u:\RR \to \RR_+$. By Lemma \ref{lem:char crit}, we
can write for any $a \in \RR$,
$$
(\psi_{ac}''(a))^{1-m} \bar\rho(a)^m =  \frac{\chi}{2} \int_{\RR} \bla \psi_{ac}''([a,b])^{1-m} \bra |a-b|^{k}\bar \rho(a) \bar \rho(b) \,  db\, ,
$$
and hence
$$
\mF[\rho] =  \frac{\chi}{2} \iint_{\RR\times\RR} \left\{\frac{1}{(m-1)} \bla \psi_{ac}''([a,b])^{1-m} \bra +\frac1k \bla \psi''([a,b]) \bra^{k}\right\}
|a-b|^{k}\bar \rho(a) \bar \rho(b) \, da db\, .
$$
Using the concavity of the power function $(\cdot)^{1-m}$ and and Lemma \ref{lem:interaction-diffdom}, we deduce
$$
\mF[\rho] \geq  \frac{\chi}{2} \iint_{\RR\times\RR} \left\{\frac{1}{(m-1)} \bla \psi''([a,b]) \bra^{1-m} +\frac1k \bla \psi''([a,b]) \bra^{k}\right\}
|a-b|^{k}\bar \rho(a) \bar \rho(b) \, da db\, .
$$
Applying characterisation \eqref{eq:charac sstates} to the energy of the stationary state $\bar \rho$, we obtain
$$
\mF[\bar \rho] =  \frac{\chi}{2} \iint_{\RR\times\RR} \left(\frac{1}{(m-1)} +\frac1k \right)   |a-b|^{k}\bar \rho(a) \bar \rho(b) \, da db\, .
$$
Since
\begin{equation}\label{rel1}
\frac{z^{1-m}}{m-1}+\frac{z^k}{k} \geq \frac{1}{m-1}+\frac{1}{k}
\end{equation}
for any real $z>0$ and for $m>m_c=1-k$, we conclude $\mF[\rho]\geq \mF[\bar \rho]$.
Equality in Jensen's inequality arises if and only if the derivative of the transport map $\psi''$ is a constant function, i.e. when $\rho$ is a dilation
of $\bar \rho$. In agreement with this, equality in \eqref{rel1} is realised if and only if $z=1$.
\end{proof}
In fact, the result in Theorem \ref{thm:HLSmequiv} implies that all critical points of $\mF$ in $\mY$ are global minimisers. Further, we obtain the
following uniqueness result:
\begin{corollary}[Uniqueness]\label{cor:!crit}
 Let $\chi>0$ and $k \in (-1,0)$. If $m_c<m$, then there exists at most one stationary state in $\mY$ to equation \eqref{eq:KS}. If $m_c<m<m^*$, then
 there exists a unique global minimiser for $\mF$ in $\mY$.
\end{corollary}
\begin{proof}
  Assume there are two stationary states to equation \eqref{eq:KS}: $\bar \rho_1, \bar \rho_2 \in \mY$. Then Theorem
  \ref{thm:HLSmequiv} implies that $\mF[\bar\rho_1]=\mF[\bar \rho_2]$, and so $\bar \rho_1$ is a dilation of $\bar \rho_2$.
 By Theorem \ref{thm:Emins}, there exists a minimiser of $\mF$ in $\mY$, which is a stationary state of equation \eqref{eq:KS} if $m_c<m<m^*$ by Theorem
 \ref{thm:minsstates}, and so uniqueness follows.
\end{proof}

Theorem \ref{thm:HLSmequiv} and Corollary \ref{cor:!crit} complete the proof of the main result Theorem \ref{thm:main2}.

\appendix
\section{Properties of the Riesz potential}\label{sec:Rieszestimate}

The estimates in Proposition \ref{prop:rieszestimate} are mainly based on the fact that the Riesz potential of a radial function can be expressed in terms of the  hypergeometric function
\begin{equation*}
F(a,b;c;z):=\frac{\Gamma(c)}{\Gamma(b)\Gamma(c-b)}\int_{0}^1(1-zt)^{-a}(1-t)^{c-b-1}t^{b-1}\,dt,
\end{equation*}
which we define for $z\in(-1,1)$, with the parameters $a,b,c$ being positive. Notice that $F(a,b,c,0)=1$ and $F$ is increasing with respect to $z\in(-1,1)$.
Moreover, if $c>1$, $b>1$ and $c>a+b$, the limit as $z\uparrow 1$ is finite and it takes the value
\begin{equation}\label{hyperlimit}
\frac{\Gamma(c)\Gamma(c-a-b)}{\Gamma(c-a)\Gamma(c-b)},
\end{equation}
see \cite[\SS 9.3]{Leb}.
 We will also make use of some elementary relations. 
First of all, there holds
\begin{equation}\label{hypertransform}
F(a,b;c;z)=(1-z)^{c-a-b} F(c-a,c-b;c;z),
\end{equation}
see  \cite[\SS 9.5]{Leb}, and it is easily seen that
\[
\frac{d}{dz}F(a,b;c;z)=\frac{ab}{c} F(a+1,b+1;c+1;z).
\]
Inserting \eqref{hypertransform} we find
\begin{equation}\label{hyperdiff}
\frac{d}{dz} F(a,b;c;z)=\frac{ab}{c} (1-z)^{c-a-b-1}F(c- a,c-b;c+1;z).
\end{equation}
To simplify notation, let us define
\begin{equation}\label{defH}
 H(a,b;c;z):= \frac{\Gamma(b)\Gamma(c-b)}{\Gamma(c)} F(a,b;c;z)
 = \int_{0}^1(1-zt)^{-a}(1-t)^{c-b-1}t^{b-1}\,dt\, .
\end{equation}

\begin{proof}[Proof of Proposition \ref{prop:rieszestimate}]
For a given radial function $\rho\in \mY$ we use polar coordinates, still denoting by $\rho$ the radial profile of $\rho$, and
compute as in \cite[Theorem 5]{TS}, see also \cite{BCLR}, \cite{Dong}  or \cite[\SS 1.3]{Dre},
\begin{equation}\label{radialrepresentation}\begin{aligned}
|x|^k\ast \rho(x)
= {\sigma_{N-1}} \int_0^{\infty} \left(\int_0^\pi
\left(|x|^2+\eta^2-2|x|\eta\cos\theta\right)^{k/2}\,\sin^{N-2}\!\theta\,d\theta\right)\,\rho(\eta) \eta^{N-1}\,d\eta\, .
\end{aligned}
\end{equation}
Then we need to estimate the integral
\begin{align}\label{esti1}
\Theta_k(r,\eta)&:= {\sigma_{N-1}}  \int_0^\pi \left(r^2+\eta^2-2r\eta \cos(\theta)\right)^{k/2} \sin^{N-2}(\theta)\, d\theta
=
\begin{cases}
 r^k \vartheta_k\left(\eta/r\right), &\eta<r\, ,\\
 \eta^k \vartheta_k\left(r/\eta\right), &r<\eta\, ,
\end{cases}
\end{align}
  where, for $u\in[0,1)$,
\begin{align*}
\vartheta_k(u)&:= {\sigma_{N-1}}  \int_0^\pi \left(1+u^2-2u \cos(\theta)\right)^{k/2} \sin^{N-2}(\theta)\, d\theta\\
&=  {\sigma_{N-1}}  \left(1+u\right)^k\int_0^\pi \left(1-4\frac{u}{(1+u)^2} \cos^2\left(\frac{\theta}{2}\right)\right)^{k/2}
\sin^{N-2}(\theta)\, d\theta\, .
\end{align*}
Using the change of variables $t=\cos^2\left(\frac{\theta}{2}\right)$, we get from the integral formulation \eqref{defH},
\begin{align}
 \vartheta_k(u)&=2^{N-2} {\sigma_{N-1}}  \left(1+u\right)^k\int_0^1 \left(1-4\frac{u}{(1+u)^2}t\right)^{k/2} t^{\frac{N-3}{2}}
 \left(1-t\right)^{\frac{N-3}{2}}\, dt\notag\\
 &= 2^{N-2} {\sigma_{N-1}}  \left(1+u\right)^k H\left(a,b;c;z\right)\label{esti2}
\end{align}
 with
 $$
 a=-\frac{k}{2}\, , \qquad b=\frac{N-1}{2}\, , \qquad c=N-1\, , \qquad z=\frac{4u}{(1+u)^2}\, .
 $$
The function $F(a,b;c;z)$ is increasing in $z$ and then for any $z\in (0,1)$ there holds
\begin{equation}\label{hypermonotone}
F(a, b;c;z)\le \lim_{z\uparrow 1} F(a,b;c;z).
\end{equation}
Note that $c-a-b=(k+N-1)/2$ changes sign at $k=1-N$, and the estimate of $\Theta_k$ depends on the sign of $c-a-b$:\\

\fbox{Case $k>1-N$}
The limit \eqref{hypermonotone} is finite if $c-a-b>0$ and it is given by the expression
\eqref{hyperlimit}.
 Therefore we get from \eqref{esti1}-\eqref{esti2} and \eqref{defH}
 $$\Theta_k(|x|,\eta)\le C_1(|x|+\eta)^{k}\le C_1|x|^{k}\qquad
\text{if} \, \, 1-N<k<0
$$ with
$C_1:=  {2^{N-2}\sigma_{N-1}}  {\Gamma(b)\Gamma(c-a-b)}/{\Gamma(c-a)}$. Inserting this into \eqref{radialrepresentation} concludes the proof of (i).\\

\fbox{Case $k < 1-N$}
If $c-a-b<0$ we  use \eqref{hypertransform}
$$
F(a,b;c;z)=(1-z)^{c-a-b} F(c-a,c-b;c;z),
$$
where now the right hand side, using \eqref{hypermonotone} and \eqref{hyperlimit}, can be bounded  from above by
$(1-z)^{c-a-b}\Gamma(c)\Gamma(a+b-c)/[{\Gamma(a)\Gamma(b)}]$ for $z\in(0,1)$. This yields from \eqref{esti1}-\eqref{esti2} and \eqref{defH} the  estimate
 \begin{equation}\label{IR1}
 \Theta_k(|x|,\eta)\le C_2|x|^{k}\left(\frac{|x|+\eta}{|x|-\eta}\right)^{1-k-N}\qquad \text{if} \, \, k<1-N
 \end{equation}
 with $C_2:=  {2^{N-2}\sigma_{N-1}}  {\Gamma(c-b)\Gamma(a+b-c)}/{\Gamma(a)}$.\\

 \fbox{Case $k=1-N$}
 If on the other hand $c-a-b=0$, we use \eqref{hyperdiff} with $c=2a=2b=N-1$, integrating it and obtaining, since $F=1$ for $z=0$,
 \[
 F(a,b;c;z)=1+\frac{N-1}4\int_{0}^z\frac{F(c-a,c-b;c+1;t)}{1-t}\,dt,
 \]
 and the latter right hand side is bounded above, thanks to \eqref{hypermonotone} and \eqref{hyperlimit}, by
 $$1+\frac{(N-1)\Gamma(N)}{4(\Gamma(N/2+1/2))^2}\,\log{\left(\frac1{1-z}\right)}$$
 for $z\in(0,1)$. This leads from \eqref{esti1}-\eqref{esti2} to
  the new estimate
  \begin{equation}\label{IR2}
  \Theta_k(|x|,\eta)\le C_2|x|^{k}\left(1+\log\left(\frac{|x|+\eta}{|x|-\eta}\right)\right) \qquad \text{if} \, \, k=1-N\, ,
  \end{equation}
with $C_2:= 2^{N-2}\sigma_{N-1}\tfrac{\Gamma\left(N/2-1/2\right)^2}{\Gamma(N-1)}\max\left\{1, \frac{(N-1)\Gamma(N)}{2\Gamma\left((N+1)/2\right)^2}\right\}$.

Now, if $\rho$ is  supported on a ball $B_R$, the radial representation \eqref{radialrepresentation} reduces to
\begin{equation}\label{compactrepresentation}
|x|^k \ast \rho(x) =\int_{0}^{R}\Theta_k(|x|,\eta)\rho(\eta)\eta^{N-1}\,d\eta,\quad x\in\mathbb{R}^N.
\end{equation}
If $|x|>R$,  we have $(|x|+\eta)(|x|-\eta)^{-1}\le (|x|+R)(|x|-R)^{-1}$ for any $\eta\in(0,R)$, therefore we can put $R$ in place of $\eta$ in the right
hand side of \eqref{IR1} and \eqref{IR2}, insert into \eqref{compactrepresentation} and conclude.
\end{proof}

\section*{Acknowledgements}
\small{We thank Y. Yao and F. Brock for useful discussion about the continuous Steiner symmetrisation. We thank X. Ros-Ot\'on, P. R. Stinga and P. Mironescu for some fruitful explanations concerning the regularity properties of fractional elliptic equations used in this work.
JAC was partially supported by the Royal Society via a Wolfson Research Merit Award and by the EPSRC grant number EP/P031587/1.
 FH acknowledges support from the EPSRC grant number EP/H023348/1 for the Cambridge Centre for Analysis. EM was partially supported by the FWF project M1733-N20. BV was partially supported by GNAMPA of INdAM, "Programma triennale della Ricerca dell'Universit\`{a} degli Studi di Napoli "Parthenope"- Sostegno alla ricerca individuale 2015-2017". EM and BV are member of the
GNAMPA group of the Istituto Nazionale di Alta Matematica (INdAM). The authors are very grateful to the Mittag-Leffler Institute for providing a fruitful working environment during the special semester \emph{Interactions between Partial Differential Equations \& Functional Inequalities}.}

\bibliographystyle{abbrv}
\bibliography{CHMVbooks}

\begin{thebibliography}{10}

\bibitem{BCLR}
D.~Balagu\'e, J.~A. Carrillo, T.~Laurent, and G.~Raoul.
\newblock Nonlocal interactions by repulsive-attractive potentials: radial
  ins/stability.
\newblock {\em Phys. D}, 260:5--25, 2013.

\bibitem{BCC}
A.~Blanchet, V.~Calvez, and J.~A. Carrillo.
\newblock Convergence of the mass-transport steepest descent scheme for the
  subcritical {P}atlak-{K}eller-{S}egel model.
\newblock {\em SIAM J. Numer. Anal.}, 46(2):691--721, 2008.

\bibitem{BCC12}
A.~Blanchet, E.~A. Carlen, and J.~A. Carrillo.
\newblock Functional inequalities, thick tails and asymptotics for the critical
  mass {P}atlak-{K}eller-{S}egel model.
\newblock {\em J. Funct. Anal.}, 262(5):2142--2230, 2012.

\bibitem{BCL}
A.~Blanchet, J.~A. Carrillo, and P.~Lauren{\c{c}}ot.
\newblock Critical mass for a {P}atlak-{K}eller-{S}egel model with degenerate
  diffusion in higher dimensions.
\newblock {\em Calc. Var. Partial Differential Equations}, 35(2):133--168,
  2009.

\bibitem{BlaCaMa07}
A.~Blanchet, J.~A. Carrillo, and N.~Masmoudi.
\newblock Infinite time aggregation for the critical {P}atlak-{K}eller-{S}egel
  model in {$\Bbb R^2$}.
\newblock {\em Comm. Pure Appl. Math.}, 61(10):1449--1481, 2008.

\bibitem{BlaDoPe06}
A.~Blanchet, J.~Dolbeault, and B.~Perthame.
\newblock Two dimensional {Keller-Segel} model in {$\RR^2$}: optimal critical
  mass and qualitative properties of the solution.
\newblock {\em Electron. J. Differential Equations}, 2006(44):1--33
  (electronic), 2006.

\bibitem{Brenier91}
Y.~Brenier.
\newblock Polar factorization and monotone rearrangement of vector-valued
  functions.
\newblock {\em Comm. Pure Appl. Math.}, 44:375--417, 1991.

\bibitem{Brockpolar}
F.~Brock and A.~Y. Solynin.
\newblock An approach to symmetrization via polarization.
\newblock {\em Trans. Amer. Math. Soc.}, 352(4):1759--1796, 2000.

\bibitem{CafSting}
L.~A. Caffarelli and P.~R. Stinga.
\newblock Fractional elliptic equations, {C}accioppoli estimates and
  regularity.
\newblock {\em Ann. Inst. H. Poincar\'e Anal. Non Lin\'eaire}, 33(3):767--807,
  2016.

\bibitem{CaCa06}
V.~Calvez and J.~A. Carrillo.
\newblock Volume effects in the {Keller-Segel} model: energy estimates
  preventing blow-up.
\newblock {\em J. Math. Pures Appl.}, 86:155--175, 2006.

\bibitem{CaCa11}
V.~Calvez and J.~A. Carrillo.
\newblock Refined asymptotics for the subcritical {K}eller-{S}egel system and
  related functional inequalities.
\newblock {\em Proc. Amer. Math. Soc.}, 140(10):3515--3530, 2012.

\bibitem{CCH1}
V.~Calvez, J.~A. Carrillo, and F.~Hoffmann.
\newblock Equilibria of homogeneous functionals in the fair-competition regime.
\newblock {\em preprint arXiv:1610.00939}.

\bibitem{CCHCetraro}
V.~Calvez, J.~A. Carrillo, and F.~Hoffmann.
\newblock The geometry of diffusing and self-attracting particles in a
  one-dimensional fair-competition regime.
\newblock {\em preprint arXiv:1612.08225}.

\bibitem{CaDo14}
J.~F. Campos and J.~Dolbeault.
\newblock Asymptotic estimates for the parabolic-elliptic {K}eller-{S}egel
  model in the plane.
\newblock {\em Comm. Partial Differential Equations}, 39(5):806--841, 2014.

\bibitem{CF}
E.~A. Carlen and A.~Figalli.
\newblock Stability for a {GNS} inequality and the log-{HLS} inequality, with
  application to the critical mass {K}eller-{S}egel equation.
\newblock {\em Duke Math. J.}, 162(3):579--625, 2013.

\bibitem{CCV}
J.~A. Carrillo, D.~Castorina, and B.~Volzone.
\newblock Ground states for diffusion dominated free energies with logarithmic
  interaction.
\newblock {\em SIAM J. Math. Anal.}, 47(1):1--25, 2015.

\bibitem{CHVY}
J.~A. Carrillo, S.~Hittmeir, B.~Volzone, and Y.~Yao.
\newblock Nonlinear aggregation-diffusion equations: Radial symmetry and long
  time asymptotics.
\newblock {\em preprint arXiv: 1603.07767}, 2016.

\bibitem{CHS17}
J.~A. Carrillo, Y.~Huang, and M.~Schmidtchen.
\newblock {\em preprint arXiv:}.

\bibitem{CLM14}
J.~A. Carrillo, S.~Lisini, and E.~Mainini.
\newblock Uniqueness for {K}eller-{S}egel-type chemotaxis models.
\newblock {\em Discrete Contin. Dyn. Syst.}, 34(4):1319--1338, 2014.

\bibitem{CLM}
P.-H. Chavanis, P.~Lauren{\c{c}}ot, and M.~Lemou.
\newblock Chapman-{E}nskog derivation of the generalized {S}moluchowski
  equation.
\newblock {\em Phys. A}, 341(1-4):145--164, 2004.

\bibitem{CM}
P.~H. Chavanis and R.~Mannella.
\newblock Self-gravitating {B}rownian particles in two dimensions: the case of
  {$N=2$} particles.
\newblock {\em Eur. Phys. J. B}, 78(2):139--165, 2010.

\bibitem{Demengel}
F.~Demengel and G.~Demengel.
\newblock {\em Functional spaces for the theory of elliptic partial
  differential equations}.
\newblock Universitext. Springer, London; EDP Sciences, Les Ulis, 2012.
\newblock Translated from the 2007 French original by Reinie Ern\'e.

\bibitem{DoPe04}
J.~Dolbeault and B.~Perthame.
\newblock Optimal critical mass in the two dimensional {Keller-Segel} model in
  $\mathbb{R}\sp 2$.
\newblock {\em C. R. Math. Acad. Sci. Paris}, 339:611--616, 2004.

\bibitem{DTGC}
P.~Domschke, D.~Trucu, A.~Gerisch, and M.~A.~J. Chaplain.
\newblock Mathematical modelling of cancer invasion: implications of cell
  adhesion variability for tumour infiltrative growth patterns.
\newblock {\em J. Theoret. Biol.}, 361:41--60, 2014.

\bibitem{Dong}
H.~Dong.
\newblock The aggregation equation with power-law kernels: ill-posedness, mass
  concentration and similarity solutions.
\newblock {\em Comm. Math. Phys.}, 304(3):649--664, 2011.

\bibitem{Dre}
I.~Drelichman.
\newblock {\em Weighted inequalities for fractional integrals of radial
  functions and applications}.
\newblock PhD thesis, Universidad de Buenos Aires, 2010.

\bibitem{EM16}
G.~Ega{\~n}a-Fern{\'a}ndez and S.~Mischler.
\newblock Uniqueness and long time asymptotic for the {K}eller-{S}egel
  equation: the parabolic-elliptic case.
\newblock {\em Arch. Ration. Mech. Anal.}, 220(3):1159--1194, 2016.

\bibitem{GC}
A.~Gerisch and M.~A.~J. Chaplain.
\newblock Mathematical modelling of cancer cell invasion of tissue: local and
  non-local models and the effect of adhesion.
\newblock {\em J. Theoret. Biol.}, 250(4):684--704, 2008.

\bibitem{HiPai01}
T.~Hillen and K.~Painter.
\newblock Global existence for a parabolic chemotaxis model with prevention of
  overcrowding.
\newblock {\em Adv. in Appl. Math.}, 26:280--301, 2001.

\bibitem{JaLu92}
W.~J\"{a}ger and S.~Luckhaus.
\newblock On explosions of solutions to a system of partial differential
  equations modelling chemotaxis.
\newblock {\em Trans. Amer. Math. Soc.}, 329:819--824, 1992.

\bibitem{KeSe70}
E.~Keller and L.~Segel.
\newblock Initiation of slime mold aggregation viewed as an instability.
\newblock {\em J. Theor. Biol.}, 26:399--415, 1970.

\bibitem{KeSe71a}
E.~Keller and L.~Segel.
\newblock Model for chemotaxis.
\newblock {\em J. Theor. Biol.}, 30:225--234, 1971.

\bibitem{KY}
I.~Kim and Y.~Yao.
\newblock The {P}atlak-{K}eller-{S}egel model and its variations: properties of
  solutions via maximum principle.
\newblock {\em SIAM Journal on Mathematical Analysis}, 44(2):568--602, 2012.

\bibitem{Leb}
N.~N. Lebedev.
\newblock {\em Special Functions and Their Applications}.
\newblock Prentice-Hall, 1965.

\bibitem{Lieb83}
E.~H. Lieb.
\newblock Sharp constants in the {Hardy-Littlewood-Sobolev} and related
  inequalities.
\newblock {\em Ann. of Math.}, 118:349--374, 1983.

\bibitem{LiebLoss}
E.~H. Lieb and M.~Loss.
\newblock {\em Analysis}, volume~14 of {\em Graduate Studies in Mathematics}.
\newblock American Mathematical Society, Providence, RI, second edition, 2001.

\bibitem{L}
P.~L. Lions.
\newblock The concentration-compactness principle in the calculus of
  variations. the locally compact case, part 1.
\newblock {\em Annales de l'I.H.P. Analyse non lineaire}, 1(2):109--145, 1984.

\bibitem{McCann95}
R.~J. McCann.
\newblock Existence and uniqueness of monotone measure-preserving maps.
\newblock {\em Duke Math. J.}, 80:309--323, 1995.

\bibitem{McCann97}
R.~J. McCann.
\newblock A convexity principle for interacting gases.
\newblock {\em Adv. Math.}, 128:153--179, 1997.

\bibitem{Mironescu}
P.~Mironescu.
\newblock Superposition with subunitary powers in {S}obolev spaces.
\newblock {\em C. R. Math. Acad. Sci. Paris}, 353(6):483--487, 2015.

\bibitem{MT15}
H.~Murakawa and H.~Togashi.
\newblock Continuous models for cell--cell adhesion.
\newblock {\em Journal of theoretical biology}, 374:1--12, 2015.

\bibitem{Nagai95}
T.~Nagai.
\newblock Blow-up of radially symmetric solutions to a chemotaxis system.
\newblock {\em Adv. Math. Sci. Appl.}, 5:581--601, 1995.

\bibitem{Nanjundiah73}
V.~Nanjundiah.
\newblock Chemotaxis, signal relaying and aggregation morphology.
\newblock {\em J. Theor. Biol.}, 42:63--105, 1973.

\bibitem{PaiHi02}
K.~Painter and T.~Hillen.
\newblock Volume-filling and quorum-sensing in models for chemosensitive
  movement.
\newblock {\em Can. Appl. Math. Q.}, 10:501--543, 2002.

\bibitem{PBSG}
K.~J. Painter, J.~M. Bloomfield, J.~A. Sherratt, and A.~Gerisch.
\newblock A nonlocal model for contact attraction and repulsion in
  heterogeneous cell populations.
\newblock {\em Bull. Math. Biol.}, 77(6):1132--1165, 2015.

\bibitem{Perthame06}
B.~Perthame.
\newblock {\em Transport equations in biology}.
\newblock Frontiers in mathematics. Birkh\"auser, 2006.

\bibitem{Ros-Oton}
X.~Ros-Oton and J.~Serra.
\newblock Regularity theory for general stable operators.
\newblock {\em J. Differential Equations}, 260(12):8675--8715, 2016.

\bibitem{Samko}
S.~G. Samko, A.~A. Kilbas, and O.~I. Marichev.
\newblock {\em Fractional integrals and derivatives}.
\newblock Gordon and Breach Science Publishers, Yverdon, 1993.
\newblock Theory and applications, Edited and with a foreword by S. M.
  Nikol$\prime$ski\u\i , Translated from the 1987 Russian original, Revised by
  the authors.

\bibitem{TS}
D.~Siegel and E.~Talvila.
\newblock Pointwise growth estimates of the {R}iesz potential.
\newblock {\em Dynamics of Continuous, Discrete and Impulsive Systems},
  5:185--194, 1999.

\bibitem{Silvestre}
L.~Silvestre.
\newblock Regularity of the obstacle problem for a fractional power of the
  {L}aplace operator.
\newblock {\em Comm. Pure Appl. Math.}, 60(1):67--112, 2007.

\bibitem{Stein}
E.~M. Stein.
\newblock {\em Singular integrals and differentiability properties of
  functions}.
\newblock Princeton Mathematical Series, No. 30. Princeton University Press,
  Princeton, N.J., 1970.

\bibitem{Villani03}
C.~Villani.
\newblock {\em Topics in optimal transportation}, volume~58 of {\em Graduate
  Studies in Mathematics}.
\newblock American Mathematical Society, Providence, RI, 2003.

\end{thebibliography}

\end{document}